\documentclass{amsart}
\usepackage{amsfonts,amsthm,amsmath}
\usepackage{amssymb}
\usepackage[usenames,dvipsnames]{color}
\usepackage[utf8]{inputenc}
\usepackage{xypic}
\usepackage{epic}
\usepackage[a4paper]{geometry}
\usepackage{tikz}
\usepackage{enumerate}
\usepackage[normalem]{ulem} 
\usepackage{hyperref}

\numberwithin{equation}{section}
\numberwithin{equation}{subsection}
\theoremstyle{plain}
	\newtheorem{thm}[equation]{Theorem}
	
	\newtheorem{lemma}[equation]{Lemma}
	\newtheorem{prop}[equation]{Proposition}
	\newtheorem{cor}[equation]{Corollary}
\theoremstyle{remark}
	\newtheorem{remark}[equation]{Remark}
	\newtheorem*{remark*}{Remark}
	
\theoremstyle{definition}
		
\newcommand{\labelpar}{\label}

\title{
Non-normal affine monoids, modules and Poincar\'e series of plumbed 3-manifolds}

\author{Tam\'as L\'aszl\'o$^{1,2}$}
\author{Zsolt Szil\'agyi$^{2}$}

\address{$^1$BCAM - Basque Center for Applied Mathematics\\
Mazarredo, 14 E48009 Bilbao, Basque Country – Spain \\}
\email{tlaszlo@bcamath.org}

\address{$^2$Alfr\'ed R\'enyi Institute of Mathematics, Hungarian Academy of Sciences,  
1053 Budapest, Re\'altanoda u. 13-15,  Hungary.}
\email{szilagyi.zsolt@renyi.mta.hu}



\begin{document}
\maketitle

\pagestyle{myheadings} \markboth{{\normalsize
T. L\'aszl\'o and Zs. Szil\'agyi}}{{\normalsize  Non-normal affine monoids, modules and Poincar\'e series}}

\begin{abstract}
We construct a non-normal affine monoid together with its modules associated with a negative definite plumbed $3$-manifold $M$. In terms of their structure, we describe the $H_1(M,\mathbb{Z})$-equivariant parts of the topological Poincar\'e series. In particular, we give combinatorial formulas for the Seiberg--Witten invariants of $M$ and for polynomial generalizations defined in \cite{LSz}.    
\end{abstract}

\section{Introduction}

\subsection{} Let $M$ be a closed oriented plumbed 3-manifold associated with a connected negative definite plumbing graph $\Gamma$. Assume that $M$ is a rational homology sphere, ie. $\Gamma$ is a tree and all the plumbed surfaces have genus zero. Let $\mathcal{V}$ be the set of vertices of $\Gamma$, $\delta_v$ is the valency of the vertex $v\in\mathcal{V}$, and $\mathcal{N}$ is the set of nodes consisting of vertices with $\delta_v\geq 3$. We consider the plumbed 4-manifold $\widetilde{X}$ associated with $\Gamma$. Its second homology $L$ is freely generated by the classes of 2-spheres $\{E_v\}_{v\in\mathcal{V}}$, its second cohomology $L'$ by the (anti)dual classes $\{E_v^*\}_{v\in\mathcal{V}}$. The intersection form embeds $L$ into $L'$ and $H:=L'/L\simeq H_1(M,\mathbb{Z})$. Denote the class of $l'\in L'$ in $H$ by $[l']$. 

We consider the (equivariant) topological Poincar\'e series $Z_{H}(\mathbf{t})=\sum_{l'\in L'}\mathbf{p}_{l'}\mathbf{t}^{l'}$ as the multivariable Taylor expansion at the origin of the equivariant zeta function 
$$f_H(\mathbf{t})=\prod_{v\in\mathcal{V}}(1-[E_v^*]\mathbf{t}^{E_v^*})^{\delta_v-2}$$
where $\mathbf{t}^{l'}:=\prod_{v\in\mathcal{V}}t_v^{l_v}$ for any $l'=\sum_{v\in\mathcal{V}}l_vE_v\in L'$. It has a natural decomposition $Z_H(\mathbf{t})=\sum_{h\in H}Z_h(\mathbf{t})\cdot h$, where the $h$-equivariant parts are given by $Z_h(\mathbf{t})=\sum_{[l']=h}p_{l'}\mathbf{t}^{l'}$ (cf. Section \ref{ss:TPs}).

\subsection{} The topological Poincar\'e series appeared in several articles studying invariants of normal surface singularities, cf. \cite{CDGEq,CDGPs,NPS,NCL,Baldur}. It is a topological tool giving `some' information on analytic invariants associated with the singularity. For special classes of singularities (see \cite{CDG99,NPS,NCL}), this series coincides with its `analytic counterpart' (cf. Section \ref{ss:mot}). This strong parallelism makes interesting connections between the geometry and topology of singularities. Moreover, motivates the following completely topological question as well: how one can recover invariants of $M$ from the associated topological Poincar\'e series?

In the spirit of the aforementioned parallelism, \cite{NJEMS} proved that the Seiberg--Witten invariants of $M$ can be given as multivariable periodic constants of the series (see Section \ref{ss:sw} for details). This result was reinterpreted in \cite{LN} by showing that the Seiberg--Witten invariants appear as coefficients of a multivariable Ehrhart polynomial of certain polytope associated with the manifold $M$. One of the  outcomes of this interpretation is a reduction for the number of variables of the series and defines reduced series $Z_h(\mathbf{t}_{\mathcal{N}})$ for any $h\in H$, reflecting on the complexity of the manifold $M$ (see Sections \ref{ss:redtps} and \ref{ss:reductionb}). Very recently, a decomposition of the topological Poincaré series into polynomial and `negative degree' parts was given in \cite{LSz}, providing an effective calculation of the Seiberg--Witten invariants from the reduced series.

\subsection{} The main motivation of our work can be explained by the following comparison.  
In \cite{LSz}, the polynomial part of $Z_h(\mathbf{t}_{\mathcal{N}})$ for a fixed $h\in H$, in particular the corresponding Seiberg--Witten invariant of $M$, is computed from the equivariant zeta function $f_H(\mathbf{t})$ by summing up equivariant polynomials associated with vertices and edges of the graph $\Gamma$ and then considering the equivariant parts of the result. In fact, the reduction helps in minimizing the number of polynomials which have to be summed up.

In this article we use another approach, extending the idea of \cite{LN}. First, we decompose the equivariant Poincar\'es series into its equivariant parts and reduce the number of variables to get $Z_{h}(\mathbf{t}_{\mathcal{N}})$. Then the reduced series can be presented as a (non-equivariant) rational fraction. Finally, we compute the desired polynomial part by summing up certain polynomial datas associated with the rational fraction and the graph $\Gamma$. The fundamental difference between the two approaches is when the equivariant decomposition happens.

This approach leads to the construction of the sets $\mathcal{M}_{a}$ indexed by some `lifts' $a$ associated with $h\in H$. There is a distinguished set $\mathcal{M}_{0}$ corresponding to the class $0\in H$ which has the structure of a non-normal affine monoid. In general, $\mathcal{M}_{a}$ is an $\mathcal{M}_{0}$-module of rank equals with the number of nodes. 
We study their structure by describing the set of holes $\overline{\mathcal{M}}_{a}\setminus \mathcal{M}_{a}$, where $\overline{\mathcal{M}}_{a}$ is the normalization of $\mathcal{M}_{a}$, see Section \ref{s:monoid}.

First of all, we describe the equivariant decomposition of the topological Poincar\'e series and express the $h$-equivariant parts $Z_h(\mathbf{t}_{\mathcal{N}})$ as binomial sums of (fine) Hilbert series associated with certain filtered pieces of the modules $\mathcal{M}_{a}$, which in fact motivates their construction. Then we give a rational representation of $Z_h(\mathbf{t}_{\mathcal{N}})$ using  generating functions associated with the holes of $\mathcal{M}_a$. 

We derive from this representation a combinatorial formula for the polynomial part of $Z_h(\mathbf{t}_{\mathcal{N}})$, in particular for the Seiberg--Witten invariant, in terms of the polynomial data given by the holes of $\mathcal{M}_{a}$. This is a generalization of the formula found in \cite[Section 7.4]{LN} for special cases.

On the other hand, we would like to emphasize that presumably the structure of the non-normal affine monoid and its modules will be an important tool studying the connections between the topology and geometry of normal surface singularities. In order to support this, we show that our construction  applied to the minimal embedded resolution graph of an irreducible plane curve singularity gives back the (analytic) semigroup associated with the singularity. This makes possible to think about a parallelism between modules with Seiberg--Witten invariants associated with links of normal surface singularities and semigroups with delta invariants associated with plane curve singularities. 
Moreover, we also discuss the role of the numerical semigroup of Seifert homology spheres given by our construction, mentioning results from \cite{CK} and \cite{Semigp}.

\subsection{}The structure of the article is the following: in Section \ref{s:lgi} we explore the necessary technical ingredients related to the combinatorics of the graph $\Gamma$, such as generalized Seifert invariants and explicit description of the generators, relations and lifts of the group $H$. Section \ref{s:tps} contains some details and preliminary results about topological Poincar\'e series, and we discuss the equivariant decomposition with explicit calculation of the reduced series. Section \ref{s:monoid} gives the construction and study the structure of the non-normal affine monoid and its modules associated with $M$. Section \ref{ss:Zrational} deduces the rational representation of the series which induces the formulas for polynomial parts and Seiberg--Witten invariants presented in Section \ref{s:SW}. Notice that we illustrate the theory of the article on examples presented in Sections \ref{example},  \ref{ex:p2}, \ref{ex:p3} and \ref{ex:nontriv}. Finally, Section \ref{s:exanal} illustrates the role of our construction by discussing two examples: semigroups of plane curve singularities and semigroups of Seifert homology spheres.

\subsection*{Acknowledgements}
TL is supported by ERCEA Consolidator Grant 615655 – NMST and also
by the Basque Government through the BERC 2014-2017 program and by Spanish
Ministry of Economy and Competitiveness MINECO: BCAM Severo Ochoa
excellence accreditation SEV-2013-0323. Partial financial support to TL was provided by the National Research, Development and Innovation Office – NKFIH under grant number  112735. Financial support to ZsSz was provided by the `Lendület' program.

\section{ Links and graphs of normal surface singularities } \labelpar{s:lgi}

Let  $(X,o)$ be a complex normal surface singularity and we consider the minimal good resolution $\pi:\widetilde{X}\to X$ with dual graph
$\Gamma$. That is, in the exceptional divisor every $(-1)$-curve must intersect at least three other irreducible components. 
We assume that the link $M$ is a rational homology sphere. Hence, $\Gamma$ is a tree and all the irreducible exceptional divisors have genus $0$.

\subsection{Arithmetics of the graph $\Gamma$}\label{sss:notgraph} 
Let $\mathcal{V}$ be the set of vertices of $\Gamma$  and denote $\delta_v$ the valency of the vertex $v$, ie. the number of edges adjacent to it.
We consider two subsets of $\mathcal{V}$: the set of \emph{nodes} $\mathcal{N}:=\{n\in \mathcal{V} \,|\, \delta_{n} \geq 3\}$ and the set of \emph{ends} $\mathcal{E} :=\{ u \in \mathcal{V} \,|\, \delta_{u} = 1\}$. The {\it chains} are connected components of the graph obtained by deleting the nodes and their adjacent edges. {\it Legs} are chains containing an end-vertex. 
For any $n\in \mathcal{N}$ we denote by $\mathcal{N}_{n}$ the set of nodes which are connected to $n$ by a chain. 
Similarly, let $\mathcal{E}_{n}$ be the set of ends connected to $n$ by their legs. Their cardinalities will be $\delta_{n,\mathcal{N}} =  |\mathcal{N}_{n}|$ and 
$\delta_{n,\mathcal{E}}= |\mathcal{E}_{n}|$. Hence, 
$
\delta_{n} = \delta_{n,\mathcal{N}} + \delta_{n,\mathcal{E}}.
$
We also distinguish the subset of (higher complexity) nodes $\widehat{\mathcal{N}} := \{n\in \mathcal{N}  \,|\, \delta_{n,\mathcal{N}} \geq 2\}$.

Since $\Gamma$ is a tree, for any two vertices $v,w\in \mathcal{V}$ there is a unique minimal connected subgraph 
$[v,w]$ with vertices $\{v_{i}\}_{i=0}^{k}$ such that $v=v_{0}$ and $w=v_{k}$. Similarly, we also introduce notations 
$[v,w)$, $(v,w]$ and $(v,w)$ for the complete subgraphs with vertices $\{v_{i}\}_{i=0}^{k-1}$, $\{v_{i}\}_{i=1}^{k}$ and $\{v_{i}\}_{i=1}^{k-1}$ respectively.

\subsubsection{}\labelpar{sss:lattices}
We consider the lattice $L := H_2 ( \widetilde{X},\mathbb{Z} )$, which is freely generated by the classes of the irreducible exceptional divisors $\{E_v\}_{v\in\mathcal{V}}$. This lattice comes with a nondegenerate negative definite intersection form $I:=[(E_v,E_w)]_{v,w}$. 
 The vertex $v$ of $\Gamma$ is decorated with $b_v:=I_{vv} \in \mathbb{Z}_{<0}$. Moreover, in the case $\delta_v\leq 2$ we have $b_v\leq -2$, since $\Gamma$ is the minimal good resolution graph. 
Set $L':=H^2( \widetilde{X}, \mathbb{Z} )$. Then the intersection form provides an embedding $L \hookrightarrow L'$ with finite factor $H:=L'/L\simeq H^2(\partial \widetilde{X},\mathbb{Z})\simeq H_1(M,\mathbb{Z})$, and it extends to $L'$ (since $L'\subset L\otimes \mathbb{Q}$).
Hence, $L'$ is the dual lattice, freely generated by the (anti-)duals $\{E_v^*\}_{v\in \mathcal{V}}$, where we
prefer the convention $( E_v^*, E_w) =  -1$ for $v = w$, and
$0$ otherwise. The class of an element $l'\in L'$ will be denoted by $[l']\in H$.

\subsubsection{}\labelpar{ss:determinants}
The {\it determinant} of a subgraph $\Gamma'\subseteq \Gamma$ is defined as the determinant of the negative of the submatrix of $I$ with rows and columns indexed with vertices of $\Gamma'$, and it will be denoted by $\mathrm{det}_{\Gamma'}$. In particular, $\mathrm{det}_\Gamma:=\det(-I)=|H|$. 
The inverse of $I$ has entries
 $(I^{-1})_{vw}=(E_v^*,E^*_w)$, all of them are negative. Moreover,
 they can be computed as  (cf. \cite[page 83 and \S 20]{EN})
 \begin{equation}\label{eq:DETsgr}
- (E_v^*,E^*_w) = \frac{\det_{\Gamma\setminus [v,w]}}{\det_{\Gamma}}.
\end{equation}
We set the following `edge' cases: $\mathrm{det}_{[v,v)}=\det_{(v,v]}=1$, $\mathrm{det}_{(v,v)}=0$ and $\mathrm{det}_{(v,v')}=1$ whenever $v$ and $v'$ are consecutive vertices. 
Then we have the following determinantal relation.
\begin{lemma}\labelpar{lem:chaineq}
Let $v,w,w',v'$ be (not necessarily distinct) vertices on a chain in the given order. Then one has
$
\mathrm{\det}_{[v,w')}\mathrm{\det}_{(w,v']} = \mathrm{\det}_{[v,v']} \mathrm{\det}_{(w,w')} + \mathrm{\det}_{[v,w)} \mathrm{\det}_{(w',v']}.
$
\end{lemma}
The proof is based on standard calculations, therefore we omit from here. Special cases are proved in \cite[10.2]{NOSZ}.

\subsubsection{\bf Generalized Seifert invariants}
Some determinants associated with subgraphs will play a special role in the sequel and they are defined as follows.
Recall that for a subgraph
\begin{center}
\begin{picture}(170,35)(80,5)
\put(100,20){\circle*{3}}
\put(130,20){\circle*{3}}
\put(200,20){\circle*{3}}
\put(230,20){\circle*{3}}
\put(100,20){\line(1,0){50}}
\put(230,20){\line(-1,0){50}}
\put(165,20){\makebox(0,0){$\cdots$}}
\put(100,30){\makebox(0,0){$-k_1$}}
\put(130,30){\makebox(0,0){$-k_2$}}
\put(230,30){\makebox(0,0){$-k_s$}}
\put(200,30){\makebox(0,0){$-k_{s-1}$}}
\put(100,10){\makebox(0,0){$v_1$}}
\put(130,10){\makebox(0,0){$v_2$}}
\put(230,10){\makebox(0,0){$v_s$}}
\put(200,10){\makebox(0,0){$v_{s-1}$}}
\end{picture}
\end{center}
with vertices $\{v_i\}_{i=1}^s$ and  $k_i\geq 2$ for all $i$
the arithmetical properties of the graph can be encoded by the normalized Seifert invariant $(\alpha,\omega)$, where $0<\omega<\alpha$ and $\mathrm{gcd}(\alpha,\omega)=1$, using Hirzebruch/negative continued fraction expansion 
$$ \alpha/\omega=[k_{1},\ldots, k_s]=
k_{1}-1/(k_{2}-1/(\cdots -1/k_{s})\cdots ).$$
In particular, the plumbed 3-manifold associated with the above graph itself is the lens space $L(\alpha,\omega)$.
We also consider $\widetilde\omega$ satisfying $\omega\widetilde\omega\equiv 1$ (mod $\alpha$), $0< \widetilde\omega<\alpha$. Clearly, these invariants are graph determinants, namely $\alpha=\mathrm{det}_{[v_1,v_s]}$, $\omega=\mathrm{det}_{[v_2,v_s]}$, $\widetilde\omega=\mathrm{det}_{[v_1,v_{s-1}]}$. Moreover, $\omega\widetilde\omega=\alpha \tau+1$ for $\tau = \mathrm{det}_{[v_2,v_{s-1}]}$ by Lemma  \ref{lem:chaineq}.

Similarly to the case of star-shaped plumbing graphs, ie. $M$ is a Seifert 3-manifold (cf. \cite{JN83,NOSZ}), we encode the information of $\Gamma$ by the normalized Seifert invariants of the chains and legs and the orbifold Euler numbers attached to the nodes $n\in \mathcal{N}$. In fact, for any $n\in \mathcal{N}$ it is convenient to consider the maximal star-shaped subgraphs $\Gamma_n$ of $\Gamma$ 
which contains only one node $n\in \mathcal{N}$: 

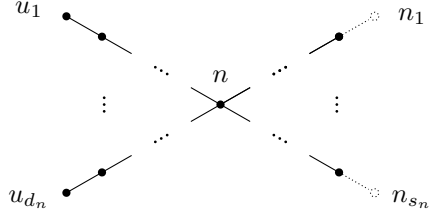
\begin{center}
\def\leg rotated by (#1){\draw[rotate around={#1:(0,0)}] (0,0) -- (0.5,0);
\fill[rotate around={#1:(0,0)}] (0.9, 0) circle [radius=0.02];
\fill[rotate around={#1:(0,0)}] (1, 0) circle [radius=0.02];
\fill[rotate around={#1:(0,0)}] (1.1, 0) circle [radius=0.02];\draw[rotate around={#1:(0,0)}] (1.5, 0) -- (2,0);
\fill[rotate around={#1:(0,0)}] (2, 0) circle [radius = 0.06];
} 

\def\endleg rotated by (#1) with label (#2) at (#3){\leg rotated by ({#1});
\draw[ rotate around={{#1}:(0,0)}] (2,0) -- (2.6,0);
\fill[rotate around={{#1}:(0,0)}] (2.6,0) circle [radius = 0.06];
\draw[rotate around={{#1}:(0,0)}] (3,0) node[#3]{#2};
}

\def\nodeleg rotated by (#1) with label (#2) at (#3){\leg rotated by ({#1});
\draw[densely dotted, rotate around={{#1}:(0,0)}] (2,0) -- (2.5,0);
\draw[densely dotted, rotate around={{#1}:(0,0)}] (2.6,0) circle [radius = 0.06];
\draw[rotate around={{#1}:(0,0)}] (3.3,0) node[#3]{#2};
}
\def\vertdots (#1){
\fill (#1, 0.1) circle [radius=0.02];
\fill (#1, 0) circle [radius=0.02];
\fill (#1, -0.1) circle [radius=0.02];}

\begin{tikzpicture}[scale=0.9]
\fill (0,0) circle [radius=0.06];
\draw (0,0.2) node[above]{$n$}; 
\leg rotated by ({30});
\nodeleg rotated by (30) with label () at (below);
\draw (2.8, 1.3) node{$n_{1}$}; 
\nodeleg rotated by (-30) with label () at (above);
\draw (2.8, -1.4) node{$n_{s_{n}}$}; 
\endleg rotated by (150)  with label () at (below);
\draw (-2.8, -1.4) node{$u_{d_{n}}$}; 
\draw (-2.8, 1.4) node{$u_{1}$}; 
\endleg rotated by (210)  with label () at (above);
\vertdots (1.7);
\vertdots (-1.7);
\end{tikzpicture}
\end{center}
where $n_i\in \mathcal{N}_n$ $(1\leq i\leq s_n)$ and $u_j\in\mathcal{E}_n$ $(1\leq j\leq d_n)$. 
We denote the normalized Seifert invariants of the legs $(n,u_j]$ by $(\alpha_{u_j},\omega_{u_j})$. 
Moreover, the chain $(n,n_i)$ connecting the nodes $n$ and $n_i$ considered as leg of $\Gamma_n$ has normalized Seifert invariant $(\alpha_{n,n_i},\omega_{n,n_i})$, while considered as leg of $\Gamma_{n_i}$ has invariant $(\alpha_{n_i,n},\omega_{n_i,n})$ with  relation $\omega_{n,n_i}\omega_{n_i,n} = \alpha_{n, n_{i}} \tau_{n,n_{i}}+1$.
Notice that $\alpha_{n,n_i}=\alpha_{n_i,n}$ and $\tau_{n,n_{i}} = \tau_{n_{i}, n}$.

\subsubsection{\bf Orbifold intersection matrix}
We define the orbifold Euler number of the star-shaped subgraph $\Gamma_n$ by
$$e_n=b_n+\sum_{v\in\mathcal{E}_n}\frac{\omega_v}{\alpha_v}+\sum_{n'\in\mathcal{N}_n}\frac{\omega_{n,n'}}{\alpha_{n,n'}}.$$
Notice that $e_n<0$ for any $n\in \mathcal{N}$, since $\Gamma_n$ itself is negative definite being a subgraph of a negative definite graph $\Gamma$. 
One can collect  these informations by defining the {\em orbifold intersection matrix} $I^{orb}= (I^{orb} _{n,n'})_{n,n'\in\mathcal{N}}$ associated with $\Gamma$ as 
$$
I^{orb}_{n,n'} 
:= 
\begin{cases}
\displaystyle e_n & \textnormal{if }n=n',
\\
\displaystyle\frac{1}{\alpha_{n,n'}} & \textnormal{if }n'\in \mathcal{N}_{n},
\\
0 & \textnormal{otherwise}.
\end{cases}
$$ 
Then $I^{orb}$ is again negative definite and  $\mathrm{det}_{\Gamma} = \mathrm{det}(-I^{orb})\cdot \mathrm{det}_{\Gamma \setminus \mathcal{N}} $ by  \cite[Lemma 4.1.4]{BNnewt}. Furthermore, one has
\begin{equation}\label{eq:IorbInverse}
(E_n^*,E_{n'}^*)=(I^{orb})^{-1}_{n,n'}.
\end{equation}

\subsection{The group $H$: generators and relations}
We introduce a partial order on $\mathcal{N}$ such that for any  two nodes $n,n'\in\mathcal{N}$ connected by a chain we pick $n<n'$ or $n'<n$. Using this partial order for any $n<n'$ we denote by $n_{n'}$ the vertex on the chain $(n,n')$ such that $(E_{n},E_{n_{n'}})=1$. When $(n,n')=\emptyset$ then $n_{n'}=n'$. 
In the sequel, we will need the following identities expressing relations regarding $E^*_v$'s using determinants of respective subgraphs of $\Gamma$. They were considered in \cite[Lemma 7.1.2]{LN} too, hence we omit their proof.

\begin{lemma}\labelpar{lem:localeq}
Assume we have a connected negative definite plumbing tree $\Gamma$ as in \ref{sss:notgraph}.
\begin{enumerate}[(a)]
\item\label{local:a} Let $n\in \mathcal{N}$ and $u\in\mathcal{E}_{n}$. Then for any $v\in [n, u)$ we have 
  $$
  E^*_{v} 
  = 
  \mathrm{det}_{(v,u]} E^*_u - \sum_{w \in (v,u]} \mathrm{det}_{(v,w)} E_{w}. 
$$   
In particular, we have $E^{*}_{n} = \alpha_{u} E^{*}_{u} - \sum_{w \in (n,u]} \mathrm{det}_{(n,w)} E_{w}$ and  $E^{*}_{v} = \omega_{u} E^{*}_{u} - \sum_{w \in (v,u]} \mathrm{det}_{(v,w)} E_{w} $ for  the first vertex $v$ on $(n,u)$.

\item\label{local:b} Consider two nodes $n<n'$ connected by a non-empty chain. Then for all $v\in [n',n_{n'})$ we have  
  $$
  E^{*}_{v} = \mathrm{det}_{(v,n_{n'}]}E^{*}_{n_{n'}} - \mathrm{det}_{(v,n_{n'})}E^{*}_{n} - \sum_{w\in (v,n_{n'}] } \mathrm{det}_{(v,w)}E_{w}.
  $$
In particular, for the vertex $m$  on the chain $(n,n')$ such that $(E_{m},E_{n'})=1$ we get $E^{*}_{m} = \omega_{n',n}E^{*}_{n_{n'}} - \tau_{n',n} E^{*}_{n} - \sum_{w\in (m,n_{n'}] } \mathrm{det}_{(m,w)}E_{w}$.  
\end{enumerate}
\end{lemma}

\subsubsection{}\labelpar{ss:lift} We introduce short notation $g_{v}:=[E^{*}_{v}]$ for classes of dual basis elements in $H$. With these notation the class of every $l'=\sum_{v\in \mathcal{V}}l'_{v}E^{*}_{v}$ can be written as
$$
[l'] = \sum_{n\in \mathcal{N}} \Big( 
a_{n}g_{n} + \sum_{u\in \mathcal{E}_{n}} a_{u}g_{u} + \sum_{n'>n} a_{n_{n'}}g_{n_{n'}} \Big)
$$
with
$$
a_{n} := l'_{n} - \sum_{\substack{n'>n\\v\in (n',n_{n'})}} l'_{v}\mathrm{det}_{(v,n_{n'})}, 
\quad
a_{u} := l'_{u} + \sum_{v\in (n,u)} l'_{v} \mathrm{det}_{(v,u]},
\quad
a_{n_{n'}} := l'_{n_{n'}} + \sum_{\substack{n'>n\\
v\in (n',n_{n'})}} l'_{v}\mathrm{det}_{(v,n_{n'}]}
$$
by Lemma \ref{lem:localeq}. 
We call $a=\sum_{n\in \mathcal{N}}a_{n}E^{*}_{n}+\sum_{u\in \mathcal{E}} a_{u}E^{*}_{u} + \sum_{n<n'}a_{n_{n'}}E^{*}_{n_{n'}}$ the \emph{reduced transform} of $l'$ 
and a \emph{reduced lift} of $[l']\in H$. 
Thus, using the idea of Neumann  \cite{neumannA} and \cite[7.1]{LN}, the group $H=L'/L$ can be presented  with generators
$$
g_{n} = [E_{n}^{*}],\ n\in \mathcal{N},\qquad g_{u} = [E_{u}^{*}],\ u\in \mathcal{E}, \qquad g_{n_{n'}} = [E^{*}_{n_{n'}}],\ n<n'
$$
and relations
\begin{flalign*}
R_{u}:=g_{n} - \alpha_{u} g_{u} 
&= 0, 
\qquad
(u\in \mathcal{E}_{n})
\\
R_{n_{n'}}:= -\alpha_{n,n'} g_{n_{n'}} + \omega_{n,n'} g_{n} + g_{n'} 
& = 0, 
\qquad
(n<n')
\\
 R_{n}:= -b_{n}g_{n} - \sum_{u\in \mathcal{E}_{n}} \omega_u g_{u} -\sum_{n<n'}g_{n_{n'}} - \sum_{n'<n} \big( \omega_{n,n'}g_{n'_{n}} - \tau_{n,n'} g_{n'} \big) 
& = 0, \qquad
( n\in \mathcal{N} ).  
\end{flalign*}
In particular, $a=\sum_{v}a_{v}E^{*}_{v}$ and $x=\sum_{v}x_{v}E^{*}_{v}$ are reduced lifts of the same group element $h\in H$ if and only if there are $\ell_{n}, \ell_{u}, \ell_{n_{n'}} \in \mathbb{Z}$ such that 
\begin{align*}
x_{u} - a_{u} 
&=
- \ell_{u}\alpha_{u} - \omega_u \ell_{n}, 
&
(u\in \mathcal{E}_{n})  
\tag{$\textnormal{R}'_{u}$} \label{R'u}
\\
x_{n} - a_{n} 
&=
\sum_{u\in \mathcal{E}_{n}} \ell_{u} - b_{n}\ell_{n} + \sum_{n>n'}\ell_{n_{n'}} + \sum_{n<n'} \omega_{n,n'} \ell_{n_{n'}} + \sum_{n<n'} \tau_{n,n'} \ell_{n'},
\tag{$\textnormal{R}'_{n}$} 
& (n\in \mathcal{N}) \label{R'n}
\\
 x_{n_{n'}}- a_{n_{n'}}
&=
-\ell_{n} - \omega_{n',n}\ell_{n'} - \alpha_{n,n'}\ell_{n_{n'}},
&
(n<n').
\tag{$\textnormal{R}'_{n_{n'}}$} \label{R'nn'}
\end{align*}
\begin{remark}
For any class $h\in H$ we will consider the unique representative $r_h\in L'$ characterized by $r_h\in \sum_{v}[0,1)E_v$ and $[r_h]=h$ (cf. \cite[5.4]{NOSZ}). In general, $r_{h}$ is not a reduced lift of $h$.
\end{remark}

\begin{remark}\label{rk:smallergen}
The group $H$ can be generated by elements $g_{u}$, $u\in \mathcal{E}$. Indeed, we choose a special partial order on nodes of the graph $\Gamma$ in the following way. We fix a `root' node and we orient all of its outgoing paths away from the root. Then the partial order is defined such that the tail of an oriented chain is greater then its head. We proceed with induction on the number of the nodes of the graph. In every step we choose a  minimal node $n$. If $\mathcal{E}_{n}\neq \emptyset$ then we express $g_{n}$ from a relation $R_{u}$, $u\in \mathcal{E}_{n}$, otherwise there must be $\widetilde{n}<n$ in the original graph $\Gamma$, hence we express $g_{n}$ from the relation $R_{\widetilde{n}_{n}}$. Moreover, either we have a unique node $n'>n$ and in this case $g_{n_{n'}}$ is expressed from $R_{n_{n'}}$, or $n$ is the root node which finishes the induction. At the end of each step we remove the chain $[n,n')$ together with all legs attached to $n$ in order to get a new graph with one node less.
\end{remark}

\subsubsection{\bf Projections}
In Section \ref{ss:redtps} we will consider the reduction of the Poincar\'e series to only node variables, which involves the use of the projection $\pi_{\mathcal{N}}:\mathbb{R}\langle E_{v} \rangle_{v\in \mathcal{V}} \to \mathbb{R}\langle E_{n} \rangle_{n\in \mathcal{N}}$ along the subspace spanned by $E_{v}$ for $v\notin \mathcal{N}$. Therefore, we introduce the projection $\mathbf{c}_{a}:=\pi_{\mathcal{N}}(a)$ of a reduced lift $a$. One can express $\mathbf{c}_{a}$ with the basis $\{\pi_{\mathcal{N}}(E^{*}_{n})\}_{n\in \mathcal{N}}$ as $
\mathbf{c}_{a} = \sum_{n\in \mathcal{N}}A_{n} \pi_{\mathcal{N}}(E^{*}_{n}) 
$ so that 
\begin{equation}\label{eq:An}
A_{n} = a_{n} + \sum_{u\in { \mathcal{E}_n}} \frac{a_{u}}{\alpha_{u}} + \sum_{n'>n} \frac{\omega_{n,n'}}{\alpha_{n,n'}}a_{n_{n'}} + \sum_{n'<n}\frac{1}{\alpha_{n',n}}a_{n'_{n}}.
\end{equation}
Furthermore, in terms of basis $\{E_{n}\}_{n\in \mathcal{N}}$ we can write $\mathbf{c}_{a} = \sum_{n\in \mathcal{N}} c_{n}E_{n}$ with 
\begin{equation}\label{eq:cn}
(c_{n})_{n \in \mathcal{N}} = (-I^{orb})^{-1} \cdot (A_{n})_{n\in \mathcal{N}}.
\end{equation}
Note that Lemma \ref{lem:localeq}  implies that the difference of an element $l'\in L'$ and its reduced transform $a$ lies on the sublattice $\mathbb{Z}\langle E_{v}\rangle_{v\notin \mathcal{N}}$. Thus, the projections of $l'$ and $a$ by $\pi_{\mathcal{N}}$ coincide,  hence $\mathbf{c}_{a} = \pi_{\mathcal{N}}(l')$ and we use notation $c_{n}(l')=c_{n}$ for the corresponding coefficients in the basis $\{E_{n}\}_{n\in \mathcal{N}}$. In particular, for the unique representative $r_{h}$ of $h$ we have $c_{n}(r_{h})\in [0,1)$ for all $n\in \mathcal{N}$, and these values are uniquely determined by $h$.

\section{Topological Poincar\'e series} \labelpar{s:tps}

\subsection{\bf Definitions} \labelpar{ss:TPs}
For any $l' = \sum_{v\in \mathcal{V}} l_{v}E_{v}\in L' $ we set $\mathbf{t}^{l'} = \prod_{v\in \mathcal{V}} t_{v}^{l_{v}}$ and let $\mathbb{Z}[H][[L']]$ be the submodule of formal power series $\mathbb{Z}[H][[t_{v}^{\pm 1/|H|}:v\in \mathcal{V}]]$ consisting of series $\sum_{l'\in L'} \mathbf{p}_{l'}\mathbf{t}^{l'}$, with coefficients in the group ring $\mathbb{Z}[H]$. 
We consider the equivariant zeta function  
\begin{equation}\label{eq:Zdef}
f_H(\mathbf{t}) = \prod_{v\in \mathcal{V}}(1 - g_v\mathbf{t}^{E_v^*})^{\delta_v-2},
\end{equation}
where $g_v=[E_v^*] \in H$. Its multivariable Taylor expansion at the origin $Z_{H}(\mathbf{t})$ is called the {\em topological Poincar\'e series} associated with the graph $\Gamma$  
and it has of form
$Z_H(\mathbf{t})=\sum_{l'} p_{l'}[l'] \mathbf{t}^{l'} \in \mathbb{Z}[H][[L']]$. 
It decomposes uniquely into equivariant parts $Z_{H}(\mathbf{t}) = \sum_{h\in H} Z_{h}(\mathbf{t})\cdot h$, where $Z_{h}({\mathbf{t}})=\sum_{[l']=h} p_{l'} \mathbf{t}^{l'}\in \mathbb{Z}[[L']]$. The presence of the group element $[l']$ in the coefficients of $Z_{H}(\mathbf{t})$ is superfluous since $Z_{H}(\mathbf{t})$ can be decomposed into equivariant parts solely by looking at the exponents. However, this group element will become useful when we consider the reduced Poincar\'e series. 

\subsubsection{}%
Moreover, $Z_H(\mathbf{t})$ is supported on the \emph{Lipman cone}  $$\mathcal{S}':=\{l'\in L'\,:\, (l',E_v)\leq 0 \ \mbox{for all $v$}\},$$ which is generated over $\mathbb{Z}_{\geq 0}$ by the duals $E_v^*$, hence $Z_{h}(\mathbf{t})$ is supported in $(l'+L)\cap \mathcal{S}'$, where $l'\in L'$ with $[l']=h$. We write $\mathcal{S}'_{\mathbb{R}}:=\mathcal{S}'\otimes \mathbb{R}$ for the \emph{real Lipman cone}.

\subsection{\bf Motivation and history}\labelpar{ss:mot}
The topological Poincar\'e series was introduced by the work of N\'emethi \cite{NPS}, motivated by the following fact: we may consider the equivariant divisorial Hilbert series $\mathcal{H}(\mathbf{t})$ of a normal surface singularity $(X,o)$ with fixed resolution graph $\Gamma$.
The key point connecting $\mathcal{H}(\mathbf{t})$ with the topology of the link $M$ (or the graph $\Gamma$) is introducing the series $\mathcal{P}(\mathbf{t})= -\mathcal{H}( \mathbf{t}) \cdot \prod_{v\in \mathcal{V}}(1 - t_v^{-1})\in \mathbb{Z}[[L']]$. 
Then the (non-equivariant) series $Z(\mathbf{t}):=\sum_{h\in H}Z_h(\mathbf{t})$ is the `topological candidate' for $\mathcal{P}(\mathbf{t})$. They agree for several singularities, e.g. for splice quotients (see \cite{NCL}), which contain all the rational, minimally elliptic or weighted homogeneous singularities.

More details regarding the analytic motivation can be found in \cite{CDGPs,CDGEq,NPS,NCL}.

\subsection{The reduced Poincar\'e series}\labelpar{ss:redtps}

L\'aszl\'o and N\'emethi \cite[Theorem 5.4.2]{LN} proved that from `Seiberg--Witten invariant point of view' (see Section \ref{ss:reductionb}) the number of variables can be reduced to $|\mathcal{N}|$. Therefore, we will mainly consider the \emph{reduced zeta function} and \emph{reduced Poincar\'e series} given by 
$$
f_H(\mathbf{t}_{\mathcal{N}}) = f_{H}(\mathbf{t})\mid_{t_v=1,v\notin \mathcal{N}}
\qquad\textnormal {and } \qquad
Z_h(\mathbf{t}_{\mathcal{N}}):=Z_h(\mathbf{t})\mid_{t_v=1,v\notin \mathcal{N}},
$$
respectively. 
We introduce notation $\mathbf{t}_{\mathcal{N}}^{x} := \mathbf{t}^{\pi_{\mathcal{N}}(x)}$ for any $x\in L'$.  
\begin{remark}
It is important to emphasize that the number of variables of the reduced zeta function, or series, is a `combinatorial measurement' for the complexity of $M$. Moreover, for some singularities $(X,o)$ whose link is $M$, the reduced series can be compared with other series (or invariants) giving information about the analytic type, cf. \cite{NPS}. 
\end{remark}

\subsection{The decomposition of the reduced Poincar\'e series}\labelpar{ss:eqdec}

We will use multiplicative notation for the group $H$ when we consider $\mathbb{Z}[H]$-coefficients of the zeta function $f_{H}$ and the Poincar\'e series $Z_{H}$. Note that  by Lemma \ref{lem:localeq}(\ref{local:a}) for $u\in \mathcal{E}_{n}$  we have  $(E^{*}_{u}, E^{*}_{n'}) = \alpha_{u}(E^{*}_{n}, E^{*}_{n'})$ for all $n'\in \mathcal{N}$, hence $\big( g_{u}\mathbf{t}_{\mathcal{N}}^{E^{*}_{u}} \big)^{\alpha_{u}} = g_{n}\mathbf{t}_{\mathcal{N}}^{E^{*}_{n}}$. Therefore, we can write
$$
f_{H}(\mathbf{t}_{\mathcal{N}}) 
= 
\frac{\prod_{n\in \mathcal{N}} \big( 1 - g_{n}\mathbf{t}_{\mathcal{N}}^{E^{*}_{n}} \big)^{\delta_{n}-2}}{  \prod_{u\in \mathcal{E}}(1 - g_{u}\mathbf{t}_{\mathcal{N}}^{E^{*}_{u}}) } 
=
\prod_{u\in \mathcal{E} } \Big( \sum_{x_{u}=0}^{\alpha_{u}-1} g_{u}^{x_{u}} \mathbf{t}_{\mathcal{N}}^{x_{u}E^{*}_{u}} \Big) \prod_{n\in \mathcal{N}} \Big( 1 - g_{n} \mathbf{t}_{\mathcal{N}}^{E^{*}_{n}} \Big)^{\delta_{n,\mathcal{N}}-2}. 
$$
Taking its Taylor expansion at the origin we get
\begin{align*}
Z_{H}(\mathbf{t}_{\mathcal{N}}) 
&= 
\prod_{u \in \mathcal{E}} 
\Big( \sum_{x_{u}=0}^{\alpha_{u}-1} g_{u}^{x_{u}} \mathbf{t}_{\mathcal{N}}^{x_{u}E^{*}_{u}} 
\Big)
\prod_{\delta_{n',\mathcal{N}}>2} 
\Big( 1 - g_{n'} \mathbf{t}_{\mathcal{N}}^{E^{*}_{n'}} \Big)^{\delta_{n',\mathcal{N}}-2}
\prod_{\delta_{n,\mathcal{N}}=1} 
\Big( \sum_{x_{n}\geq0} g_{n}^{x_{n}} \mathbf{t}_{\mathcal{N}}^{x_{n}E^{*}_{n}} 
\Big)
\\
&=
\sum_{x\in \mathcal{X}} 
\prod_{\delta_{n',\mathcal{N}}>2} (-1)^{x_{n'}} \binom{\delta_{n',\mathcal{N}}-2}{x_{n'}} 
\prod_{v\in \mathcal{E} \cup \mathcal{N}} g_{v}^{x_{v}} \mathbf{t}_{\mathcal{N}}^{x_{v}E^{*}_{v}},
\end{align*}
where the sum is over the set
$$
\mathcal{X}=
\left\{
x=\sum_{v\in \mathcal{N}\cup\mathcal{E}}x_{v} E^{*}_{v} \in L'\ \Bigg|\, 
\begin{array}{ll}
0 \leq x_{u} < \alpha_{u}, & u\in \mathcal{E}
\\
0\leq x_{n}, & n\notin \widehat{\mathcal{N}}
\\
0\leq x_{n'} \leq \delta_{n',\mathcal{N}} - 2, & n'\in \widehat{\mathcal{N}} 
\end{array}
\right\}.
$$
In particular, the $h$-equivariant part of $Z_{H}(\mathbf{t}_{\mathcal{N}})$ equals
\begin{equation}\label{eq:Zh}
Z_{h}(\mathbf{t}_{\mathcal{N}}) = \sum_{x\in \mathcal{X}_{h}} \prod_{\delta_{n',\mathcal{N}}>2} (-1)^{x_{n'}} \binom{\delta_{n',\mathcal{N}}-2}{x_{n'}} \cdot \mathbf{t}_{\mathcal{N}}^{x},
\end{equation}
where $\mathcal{X}_{h} = \{x\in \mathcal{X} \,|\,  [x]=h\}$.

To describe $\mathcal{X}_{h}$ more explicitely, we fix a reduced lift $a$ of $h$ (cf. Section \ref{ss:lift}) and we define an affine lattice
$$
\mathbb{Z}^{\mathcal{N}}(a) = \left\{\ell = \sum_{n\in \mathcal{N}}\ell_{n}E_{n} \in \mathbb{Z}\langle E_{n}\rangle_{n\in \mathcal{N}} \,\bigg|\,  \ell_{n}+\omega_{n',n}\ell_{n'} \equiv a_{n_{n'}}\ (\,\mathrm{mod}\ \alpha_{n,n'}), \ \forall\, n<n' \right\}.
$$
Moreover, for any $n\in \mathcal{N}$ consider the quasilinear function
\begin{equation}
\begin{split}
N_{a}(\ell,n):=a_{n} &+ \sum_{ n'>n}\frac{\omega_{n,n'}}{\alpha_{n,n'}}a_{n_{n'}}+\sum_{ n'<n}\frac{1}{\alpha_{n,n'}}a_{n_{n'}}\\
&- \Big( b_{n} + \sum_{n'\in \mathcal{N}_{n}} \frac{\omega_{n,n'}}{\alpha_{n,n'}} \Big)\ell_{n}- \sum_{n' \in \mathcal{N}_{n}} \Big( \frac{1}{\alpha_{n,n'}}\ell_{n'}\Big)+
\sum_{u\in \mathcal{E}_{n}} \Big \lfloor  \frac{a_{u} - \omega_{u}\ell_{n}}{\alpha_{u}} \Big\rfloor.
\end{split}
\end{equation}
Then we have the following parametrization of $\mathcal{X}_{h}$.

\begin{prop}\labelpar{prop:correspondence}
\begin{enumerate}[(a)]
\item\label{correspondence:a}
There is a bijection 
$$
\mathcal{S}_{a}= \left\{\ell \in \mathbb{Z}^{\mathcal{N}}(a)  \ \bigg|\, \begin{array}{ll}
0 \leq N_a(\ell,n), & 
n\notin \widehat{\mathcal{N}}
\\
0 \leq N_a(\ell,n) \leq \delta_{n,\mathcal{N}}-2, & 
n\in \widehat{\mathcal{N}}
\end{array} \right\}
\longrightarrow \mathcal{X}_{h}
$$
 given by
$
x_{n} = N_{a}(\ell,n)$ for any $n\in \mathcal{N}$ and $x_{u} = \alpha_{u} \left\{ \dfrac{a_{u}-\omega_{u}\ell_{n}}{\alpha_{u}}\right\}$  for any $u\in \mathcal{E}$.
\item\label{correspondence:b}
We have 
$$
\mathbf{t}_{\mathcal{N}}^{x} = \mathbf{t}^{\mathbf{c}_{a}+\ell} = \prod_{n\in \mathcal{N}} t_{n}^{c_{n}+\ell_{n}}.
$$ 

\item\label{correspondence:b} 
Finally, the $h$-equivariant part of the reduced Poincar\'e series equals
\begin{equation}\label{EqZ}
Z_{h}(\mathbf{t}_{\mathcal{N}})  = \sum_{\ell \in \mathcal{S}_{a}} \prod_{\delta_{n',\mathcal{N}}>2} (-1)^{N_{a}(\ell,n')} \binom{\delta_{n',\mathcal{N}}-2}{N_{a}(\ell,n')}\cdot \mathbf{t}^{\mathbf{c}_a+\ell}.
\end{equation}
\end{enumerate}
\end{prop}

\begin{remark}\

\begin{enumerate}[(a)]
 \item In fact, one can write (\ref{EqZ}) in the form 
\begin{equation}\label{EqZgen}
Z_{h}(\mathbf{t}_{\mathcal{N}})  = \sum_{\ell \in \mathcal{S}_{a}} \prod_{n\in\mathcal{N}} (-1)^{N_{a}(\ell,n)} \binom{\delta_{n,\mathcal{N}}-2}{N_{a}(\ell,n)}\cdot \mathbf{t}^{\mathbf{c}_a+\ell},  
\end{equation}
if we regard $\binom{\delta_{n,\mathcal{N}}-2}{N_{a}(\ell,n)}$ to be the generalized binomial coefficient.

\item In Section \ref{s:monoid} we give a more explicit description of the set $\mathcal{S}_{a}$, which will lead to the rational form of $Z_{h}(\mathbf{t}_{\mathcal{N}})$ in Subsection \ref{ss:Zrational}. 
\end{enumerate}
\end{remark}

\begin{proof}
\begin{enumerate}[(a)]
\item
Since both $x$ and $a$ are reduced lifts of $h$, there are $\ell_{u}, \ell_{n}, \ell_{n_{n'}}\in \mathbb{Z}$ such that (\ref{R'u}), (\ref{R'n}) and (\ref{R'nn'}) hold with $x_{n_{n'}}=0$ (cf. Section \ref{ss:lift}). We can eliminate $\ell_{u}$'s and $\ell_{n_{n'}}$'s as follows. The relation (\ref{R'nn'}) is equivalent with
\begin{equation}\label{eq:mod}
a_{n_{n'}} \equiv \ell_{n} + \omega_{n',n} \ell_{n'}\ (\textnormal{mod }\alpha_{n,n'}) 
\qquad \textnormal{and}\qquad
\ell_{n_{n'}} = \frac{a_{n_{n'}} - \ell_{n} - \omega_{n',n}\ell_{n'} }{ \alpha_{n,n'}}, \qquad (n<n').
\end{equation}
The conditions $0\leq x_{u}< \alpha_{u}$ and (\ref{R'u}) are equivalent with 
\begin{equation}\label{EqCe}
x_{u} = \alpha_{u} \cdot \left\{ \frac{a_{u} - \omega_u\ell_{n}}{\alpha_u}\right\}
\qquad \textnormal{and}\qquad
\ell_{u} = \Big \lfloor  \frac{a_{u} - \omega_u \ell_{n}}{\alpha_{u}} \Big\rfloor \in \mathbb{Z}.
\end{equation}
Moreover, by substituting $\ell_{n_{n'}}$ from (\ref{eq:mod}) and $\ell_{u}$ from (\ref{EqCe}) into (\ref{R'n}) we get $x_{n} = N_{a}(\ell,n)$. Thus, we have defined a map from $\mathcal{S}_{a} \to \mathcal{X}_{h}$.
To show that this map is invertible, note that
\begin{equation}\label{eq:Nabar}
x_{n}+\sum_{u\in \mathcal{E}_{n}} \frac{x_{u}}{\alpha_{u}} 
= 
N_{a}(\ell,n) + \left\{ \frac{a_{u} - \omega_u\ell_{n}}{\alpha_u}\right\}
=
A_{n} - e_n\ell_{n} - \sum_{n'\in \mathcal{N}_{n}}\frac{1}{\alpha_{n,n'}}\ell_{n'},
\end{equation}
where  $A_{n}$ was considered in (\ref{eq:An}) 
as the coefficient of the projection $\mathbf{c}_{a}=\pi_{\mathcal{N}}(a)$ in the $\{E^{*}_{n}\}_{n\in \mathcal{N}}$-basis. We can write (\ref{eq:Nabar}) in matrix form $\big( x_{n}+\sum_{u\in \mathcal{E}_{n}} \frac{x_{u}}{\alpha_{u}} \big)_{n} = (A_{n})_{n} - I^{orb}\cdot (\ell_{n})_{n}$, which can be reformulated simply as
\begin{equation}\label{eq:x=a-l}
\mathbf{c}_{x} = \mathbf{c}_{a} + \ell.
\end{equation} 
This shows  that $\mathcal{S}_{a}\to \mathcal{X}_{h}$ is invertible. 
\item 
Follows from (\ref{eq:x=a-l}).
\item
Follows from (\ref{eq:Zh}) and the previous two parts of the proposition.
\end{enumerate}
\end{proof}

\begin{remark}
\begin{enumerate}[(a)]
\item In the case when $M$ is a Seifert rational homology sphere, or equivalently, the graph $\Gamma$ is star-shaped, the reduced Poincar\'e series has only one variable associated with the central vertex $n$ of $\Gamma$. Since $\mathcal{S}_{a}=\{\ell=\ell_{n}E_{n}\in\mathbb{Z}E_{n} \ | \ 0\leq N_{a}(\ell,n)\}\subset (-c_{n}+\mathbb{R}_{\geq0})\cap \mathbb{Z}$, we can rephrase (\ref{EqZgen}) with  
$$Z_{h}(t_{n})=\sum_{\ell_{n}\geq -c_{n}} \max\{0,N_{a}(\ell,n)+1\}\cdot t^{c_{n}+\ell_{n}}.$$
The formula was given in this form in \cite{LN}. Similar computations for Seifert manifolds can be found also in N\'emethi and Nicolaescu \cite{NN2} and Neumann \cite{neumannA}. 

\item Assume $\Gamma$ has two nodes $n_0$ and $\widetilde n_0$ with $\delta_{n,\mathcal{N}}=1$ and the subgraph $[n_0,\widetilde n_0]$ contains all the other nodes with $\delta_{n,\mathcal{N}}=2$. Then $Z_{h}(\mathbf{t}_{\mathcal{N}})=
\sum_{\ell\in \mathcal{S}_{a}} \mathbf{t}^{\mathbf{c}_{a}+\ell}$ is the generating series of $\mathcal{S}_{a}$. In particular, this formula for the two-node case ($\widehat{\mathcal{N}}=\emptyset$) can be found in \cite[Section 7]{LN}.

\item If $\Gamma$ is the plumbing graph of the link of a 2-dimensional Newton nondegenerate hypersurface singularity, then \cite{BNnewt} implies  $\delta_{n,\mathcal{N}}\leq 3$ for any $n\in \mathcal{N}$. Hence, by (\ref{EqZ}) the coefficients of $Z_{h}(\mathbf{t}_{\mathcal{N}})$ can be  either $N_{a}(\ell,n)+1$  if $\Gamma$ is star-shaped, or $\pm 1$ otherwise. Similar consequences regarding this example can be found in \cite[Lemma 7.1.12]{Baldur}.
\end{enumerate} 
 
\end{remark}

\subsection{\bf Alternative decomposition}
\labelpar{s:alter-decomp}
We discuss an alternative decomposition of the Poincar\'e series, which emphasizes more the relation with non-normal affine monoids and also simplifies the proof of Lemma \ref{Thm2}.
Thus, we write the reduced zeta function in such a way that for every node $n$ the term $(1-g_{n}\mathbf{t}_{\mathcal{N}}^{E_{n}^{*}})$ appears in the denominator, that is
\begin{align*}
f_{H}(\mathbf{t}_{\mathcal{N}}) 
&=
\frac{
\prod_{ u \in \mathcal{E} } \Big( \sum_{x_{u}=0}^{\alpha_{u}-1} g_{u}^{x_{u}} \mathbf{t}_{\mathcal{N}}^{x_{u}E^{*}_{u}} \Big)
\prod_{ n \in \mathcal{N} } \Big( 1 - g_{n} \mathbf{t}_{\mathcal{N}}^{E_{n}^{*}} \Big)^{\delta_{n,\mathcal{N}}-1}
}{
\prod_{ n \in \mathcal{N} } \Big( 1 - g_{n} \mathbf{t}_{\mathcal{N}}^{E_{n}^{*}} \Big).
}
\end{align*}
It yields Taylor expansion at the origin
\begin{align*}
Z_{H}(\mathbf{t}_{\mathcal{N}})
&=
\prod_{ u \in \mathcal{E} } \Big( \sum_{x_{u}=0}^{\alpha_{u}-1} g_{u}^{x_{u}} \mathbf{t}_{\mathcal{N}}^{x_{u}E^{*}_{u}} \Big)
\prod_{ n \in \mathcal{N} } \Big( \sum_{k_{n}=0}^{\delta_{n,\mathcal{N}}-1} (-1)^{k_{n}} \binom{\delta_{n,\mathcal{N}}-1}{k_{n}} g_{n}^{k_{n}} \mathbf{t}_{\mathcal{N}}^{k_{n}E^{*}_{n}} \Big)
\prod_{ n \in \mathcal{N} } \Big( \sum_{x_{n}\geq 0} g_{n}^{x_{n}} \mathbf{t}_{\mathcal{N}}^{x_{n}E^{*}_{n}} \Big)
\\
&=
\sum_{\substack{0\leq k_{n} \leq \delta_{n,\mathcal{N}}-1\\ \forall n\in \mathcal{N}}}
\bigg[ \prod_{n\in \mathcal{N}}(-1)^{k_{n}}\binom{\delta_{n,\mathcal{N}}-1}{k_{n}} \bigg] \sum_{y\in \mathcal{X}(k)} \prod_{v\in \mathcal{E} \cup \mathcal{N}} g_{v}^{y_{v}} \mathbf{t}_{\mathcal{N}}^{y_{v}E^{*}_{v}},
\end{align*}
where $k=\sum_{n\in \mathcal{N}} k_{n}E^{*}_{n}$ and the latter sum runs over the index set
$$
\mathcal{X}(k) = \left\{ y=\sum_{v\in \mathcal{E} \cup \mathcal{N}}y_{v}E^{*}_{v} \ \Bigg|\ 
\begin{array}{ll}
0 \leq y_{u} < \alpha_{u}, & u\in \mathcal{E}
\\
k_{n} \leq y_{n}, & n \in \mathcal{N}
\end{array}
\right\}.
$$
For simplicity, we introduce short notations $(-1)^{k}\binom{\delta-1}{k} := \prod_{n\in \widehat{\mathcal{N}}} (-1)^{k_{n}} \binom{\delta_{n,\mathcal{N}}-1}{k_{n}}$ and $0\leq k \leq \delta-1$ instead of $0\leq k_{n} \leq \delta_{n,\mathcal{N}}-1$ for all $n\in \widehat{\mathcal{N}}$. 
Consequently, the $h$-equivariant part equals
$$
Z_{h}(\mathbf{t}_{\mathcal{N}}) =  \sum_{0\leq k \leq \delta-1} (-1)^{k} \binom{\delta-1}{k} \sum_{y\in \mathcal{X}_{h}(k)}   \mathbf{t}_{\mathcal{N}}^{y},
$$
where $\mathcal{X}_{h}(k) = \{y\in \mathcal{X}(k) \,|\, [y]=h\}$. Similarly to Proposition \ref{prop:correspondence}, we have a bijection
\begin{equation}\label{Eq:bijection2}
\mathcal{M}_{a}(k) := \left\{ \ell \in \mathbb{Z}^{\mathcal{N}}(a)  \,\big|\, N_{a}(\ell,n)\geq k_{n}, \ \forall n\in \mathcal{N} \right\} 
\longrightarrow 
\mathcal{X}_{h}(k)
\end{equation}
given by $y_{n} = N_{a}(\ell, n)$ for any $n\in \mathcal{N}$ and $y_{u} = \alpha_{u} \left\{ \dfrac{a_{u} - \omega_{u}\ell_{n}}{\alpha_{u}} \right\}$ for any $u\in \mathcal{E}$, which yields the following form of the Poincar\'e series
\begin{equation}\label{Eq:Zh-alternative}
Z_{h}(\mathbf{t}_{\mathcal{N}}) 
= 
 \sum_{0\leq k \leq\delta-1} (-1)^{k} \binom{\delta-1}{k} \sum_{\ell \in \mathcal{M}_{a}(k)} \mathbf{t}^{\ell+\mathbf{c}_{a}}.
\end{equation}

\section{Non-normal affine monoids and modules}\labelpar{s:monoid}

One can consider the `normalization' of the quasilinear function $N_{a}(\ell,n)$ by introducing the linear function
$$
\overline N_{a}(\ell,n):=N_{a}(\ell,n) + \sum_{u\in \mathcal{E}_{n}} \left\{ \frac{a_{u} - \omega_{u}\ell_{n}}{\alpha_{u}} \right\} 
\stackrel{(\ref{eq:An})}{=}A_{n} - e_{n}\ell_{n} - \sum_{n'\in \mathcal{N}_{n}} \frac{1}{\alpha_{n,n'}} \ell_{n'}.
$$
In particular, notice that for any $n\in \mathcal{N}$ with $\delta_{n,\mathcal{E}}=0$ we have equality $\overline{N}_{a}(\ell,n) = N_{a}(\ell,n)$. 

Associated with the pair $(\Gamma, a)$, we define the sets 
\vspace{0.3cm}
\begin{equation*}
 \mathcal{M}_{a}:=
\{ \ell\in \mathbb{Z}^{\mathcal{N}}(a)  \,|\,  N_{a}(\ell,n)\geq 0,\ \forall n\in \mathcal{N}\} \ \ \mbox{and} \ \ 
\overline{\mathcal{M}}_{a}:=
\{ \ell\in \mathbb{Z}^{\mathcal{N}}(a)  \,|\,  \overline{N}_{a}(\ell,n)\geq 0,\ \forall n\in \mathcal{N}\}.
\vspace{0.3cm}
\end{equation*}
From $\overline{N}_{a}(\ell,n) \geq N_{a}(\ell,n)$ follows that $\mathcal{M}_{a} \subset \overline{\mathcal{M}}_{a}$. Moreover, if we consider the real cone $\mathcal{C}^{orb} := \pi_{\mathcal{N}} (\mathcal{S}'_{\mathbb{R}}  ) = \{ \ell = \sum_{n\in \mathcal{N}}\ell_{n} E_{n}  \,|\, -I^{orb} \cdot (\ell_{n})_{n} \geq 0 \}
$ then 
\begin{align*}
\overline{\mathcal{M}}_{a} = (\mathcal{C}^{orb} - \mathbf{c}_{a}) \cap \mathbb{Z}^{\mathcal{N}}(a).
\end{align*}
\begin{remark} 
Notice that by Remark \ref{rk:smallergen} we can choose $a$ such that every $a_{n_{n'}}=0$, hence for such an $a$ the lattice $\mathbb{Z}^{\mathcal{N}}(a)$ is independent of $a$.
\end{remark}

\begin{lemma}\label{Lm1}
\begin{enumerate}[(a)]
\item\label{Lm1a}
$\mathcal{M}_{ 0}$ and 
$\overline{\mathcal{M}}_{ 0}$ are affine monoids.
$\overline{\mathcal{M}}_{0}$ is the normalization of  $\mathcal{M}_{0}$.
\item\label{Lm1b}
$\mathcal{M}_{a}$ and $ \overline{\mathcal{M}}_{a}$ are finitely generated 
$\mathcal{M}_{ 0}$-modules,
$\mathcal{M}_{a}$ is a submodule of $ \overline {\mathcal{M}}_{a}$. 
\end{enumerate}
\end{lemma}

\begin{proof}
(\ref{Lm1a}) is elementary. Part (\ref{Lm1b}) follows from \cite[Theorem 2.12]{BG}, that is 
$\overline{\mathcal{M}}_{a}$ is finitely generated over $\overline{\mathcal{M}}_{ 0}$, but $\overline{\mathcal{M}}_{ 0}$ itself is 
finitely generated as an $\mathcal{M}_{ 0}$-module.
\end{proof}

\subsection{\bf The set of holes}\labelpar{sssec:holes}

The \emph{holes} of the $\mathcal{M}_{0}$-module $\mathcal{M}_{a}$ is defined to be the set $\overline{\mathcal{M}_{a}}\setminus
\mathcal{M}_{a}$. By \cite[4.36]{BG}, this is small, in the sense that it is contained in finitely many hyperplanes parallel to the facets of 
$\mathcal{M}_{0}$, that is of $\mathcal{C}^{orb}$. More details on the decompositions for the set of holes in general can be found in \cite{BG} and \cite{K}. In the following, we describe $\overline{\mathcal{M}_{a}}\setminus
\mathcal{M}_{a}$ for the present special case.

\begin{lemma}\label{Lm2}
Fix $\kappa\geq0$ and a reduced lift $a$. Then there exists $\mathfrak{v}_{n}\in \mathcal{M}_{0}$, $n\in \mathcal{N}$ such that $\mathbb{R}_{\geq0}\langle \mathfrak{v}_{n}\rangle_{n\in \mathcal{N}} = \mathcal{C}^{orb}$ and satisfying the following properties: for any $\ell \in \overline{\mathcal{M}}_{a}$ one has
\begin{enumerate}[(a)]
\item\label{Lm2-a}
$N_{a}(\ell + \mathfrak{v}_{n},n')=N_{a}(\ell, n')$ for all $n'\neq n$,

\item\label{Lm2-b}
$N_{a}(\ell + \mathfrak{v}_{n},n)\geq \kappa$.
\end{enumerate}
\end{lemma}

\begin{proof}
We can choose $\mathfrak{v}_{n} = \lambda_{n}\pi_{\mathcal{N}}(E^{*}_{n})$ for $\lambda_{n}$ sufficiently large. Indeed, for $\mathfrak{v}_{n} = \sum_{n'\in \mathcal{N}}\mathfrak{v}_{n,n'}E_{n'} \in \mathbb{Z}^{\mathcal{N}}(0)$ such that $\{\omega_{u}\mathfrak{v}_{n,n'}/\alpha_{u}\}=0$ for any $n'\in \mathcal{N}$ and $u\in \mathcal{E}_{n'}$ we have $N_{a}(\ell+\mathfrak{v}_{n},n')=N_{a}(\ell,n')+\overline{N}_{0}(\mathfrak{v}_{n},n')$. By (\ref{eq:IorbInverse}) note that $\overline{N}_{0}(\mathfrak{v}_{n},n')=0$ for any $n'\neq n$, which implies (\ref{Lm2-a}).  Finally, we have $N_{a}(\ell, n) \geq \overline{N}_{a}(\ell, n) - |\mathcal{E}_{n}|$. Hence, if we choose $\lambda_{n}$ such that $\overline{N}_{0}(\mathfrak{v}_{n},n) = \lambda_{n}\geq |\mathcal{E}_{n}|+\kappa$ then we have $N_{a}(\ell+\mathfrak{v}_{n}, n)\geq \kappa$ for any $\ell \in \overline{\mathcal{M}}_{a}$. 
\end{proof}

\begin{remark}\label{rem:basis}
\begin{enumerate}[(a)] 
\item\label{rem:basis-a}
The vectors $\mathfrak{v}_{n}$ given in the  above proof does not depend on $a$, hence they can be chosen for all $a$. 

\item\label{rem:basis-c}
Alternatively, we can construct rather smaller vectors $\{\mathfrak{v}_n\}_{n\in\mathcal{N}}$ also satisfy (\ref{Lm2-a}) and (\ref{Lm2-b}). Require the vanishing $\{\omega_u \mathfrak{v}_{n,n'}/\alpha_u\}=0$ only for $u\in\mathcal{E}_{n'}$ $n'\neq n$, and assume $N_{0}(\mathfrak{v}_n,n)\geq 0$. Moreover, we require also that $N_{a}(\ell+\mathfrak{v}_n,n)\geq \kappa$ for any $\ell\in (\Box-\mathbf{c}_{a})\cap \mathbb{Z}^{\mathcal{N}}({a})$, where $\Box = \sum_{n\in \mathcal{N}}[0,1)\mathfrak{v}_{n}$ is the semiopen cube generated by $\{\mathfrak{v}_n\}_{n\in\mathcal{N}}$. Then one can check that these conditions imply (\ref{Lm2-a}) and (\ref{Lm2-b}).
\end{enumerate}
\end{remark}

We define the sets 
$$
\mathcal{M}_{a,n}^{-}:=\{ \ell\in (\Box - \mathbf{c}_{a}) \cap \mathbb{Z}^{\mathcal{N}}(a)  \,|\,  N_{a}(\ell,n) < 0  \}, \ \ 
\mathcal{M}^{-}_{a,\,\mathcal{I}}:= \bigcap_{n\in \mathcal{I}} \mathcal{M}^{-}_{a,\,n} \ \ \mbox{for every} \ \ \mathcal{I}\subseteq \mathcal{N}, $$ 
and let $F_{\mathcal{I}} = \mathbb{Z}_{\geq0} \langle \mathfrak{v}_{n'} \rangle_{ n'\in \mathcal{N} \setminus \mathcal{I}}$ be the `face' associated with $\mathcal{I}$.  In particular, $\mathcal{M}^{-}_{a,\, \emptyset}= (\Box - \mathbf{c}_{a}) \cap \mathbb{Z}^{\mathcal{N}}(a)$.
One can see immediately that 
$
\left\{\ell\in \mathbb{Z}^{\mathcal{N}}(a)  \,|\, N_{a}(\ell,n) < 0 \right\} 
= 
\bigsqcup_{\ell\in \mathcal{M}^{-}_{a, n}} (\ell+ F_{n}).$ Moreover, we conclude the following generalization of \cite[Proposition 7.3.5]{LN}.

\begin{prop}\label{Thm1}
Let $\{\mathfrak{v}_{n}\}_{n\in \mathcal{N}}$ as in Lemma \ref{Lm2}. Then
\begin{enumerate}[(a)]
\item\label{Thm1-a}
The normalization $\overline{\mathcal{M}}_{a}$ is given by
$$
\overline{\mathcal{M}}_{a} = \bigsqcup_{\ell\in \mathcal{M}^{-}_{a,\, \emptyset}} 
\ell + \mathbb{Z}_{\geq 0}\langle \mathfrak{v}_{n} \rangle_{ n\in \mathcal{N}}.
$$ 

\item\label{Thm1-b}
The set of holes $\overline{\mathcal{M}}_{a} \setminus \mathcal{M}_{a}$ is described by 
$$
\overline{\mathcal{M}}_{a}  \setminus \mathcal{M}_{a}
=
\bigcup_{n\in \mathcal{N}} \Big( 
\bigsqcup_{\ell\in \mathcal{M}^{-}_{a,\,n}}
\ell+F_{n}
\Big),
$$
where $\displaystyle \bigcap_{n\in \mathcal{I}} \Big( \bigsqcup_{\ell\in \mathcal{M}^{-}_{a,\,n}} \ell+F_{n} \Big) = \bigsqcup_{\ell\in \mathcal{M}^{-}_{a,\,\mathcal{I}}} \ell+F_{\mathcal{I}} $.
\end{enumerate}
\end{prop}

\begin{proof}
(a) is implied by \cite[Proposition 2.43]{BG}. Part (b) follows from the choice of $\{\mathfrak{v}_n\}_{n\in\mathcal{N}}$ and their properties from Lemma \ref{Lm2}. 
\end{proof}

\begin{cor}\label{genser}
The structure of any set $\mathcal{D}$ can be encoded by defining the generating series (fine Hilbert series) $\mathcal{H}_{\mathcal{D}}(\mathbf{t}):=\sum_{\ell\in\mathcal{D}}\mathbf{t}^{\ell}$. In the case of $\mathcal{M}_{a}$, Proposition \ref{Thm1} implies the following form:
$$\mathcal{H}_{\mathcal{M}_{a}}(\mathbf{t}_{\mathcal{N}})=
\sum_{\emptyset \subseteq \mathcal{I}\subseteq\mathcal{N}}(-1)^{|\mathcal{I}|}
\frac{\sum_{\ell \in \mathcal{M}_{a,\mathcal{I}}^{-}} \mathbf{t}^{\ell}} {\prod_{n\notin \mathcal{I}} (1-\mathbf{t}^{\mathfrak{v}_{n}})}.$$
\end{cor}

\subsection{\bf Multi-index filtration and generating series}\labelpar{sss:Hilb}
We consider a filtration  $\{\mathcal{M}_{a}(k)  \,|\, k\in \mathbb{Z}\langle E^{*}_{j}\rangle_{j\in \mathcal{J}} \}$ on the 
$\mathcal{M}_{0}$-module $\mathcal{M}_{a}$ associated with a fixed index set $\mathcal{J}\subset \mathcal{N}$ by defining the submodules 
$$
\mathcal{M}_{a}(k):=\{\ell \in \mathcal{M}_{a} \,|\,  N_{a}(\ell,n)\geq k_n,\ \forall n\in \mathcal{J}\},
$$
where $k = \sum_{j\in \mathcal{J}}k_{j}E^{*}_{j}$.

\begin{remark}\labelpar{rk:Ma(k)}
We observe that $\mathbb{Z}^{\mathcal{N}}(a)=\mathbb{Z}^{\mathcal{N}}(a-k)$ and $N_{a}(\ell,n)\geq k_{n}$ is equivalent with $N_{a-k}(\ell,n)\geq 0$ for any $n\in\mathcal{J}$,  therefore $\mathcal{M}_{a}(k)=\mathcal{M}_{a-k}$.   
\end{remark}

Clearly, $\mathcal{M}_{a}(k)\supset \mathcal{M}_{a}(k+ E^{*}_{j})$, moreover for any $\mathcal{I} \subset \mathcal{J}$ we have  
\begin{equation}\label{eq:fil1}
 \bigcap_{i\in \mathcal{I}} \mathcal{M}_{a}(k+E^{*}_{i})=\mathcal{M}_{a} ( k+  E_{\mathcal{I}}^{*} ), 
\end{equation}
where we use notation $E_{\mathcal{I}}^*:=\sum_{i\in \mathcal{I}}E^{*}_{i}$.
One can also consider the associated graded object
\begin{equation}\label{eq:fil2}
\mathrm{gr}_{k} \mathcal{M}_{a}:=\mathcal{M}_{a}(k)\setminus \bigcup_{j\in\mathcal{J}}\mathcal{M}_{a}(k+E^{*}_{j})
\end{equation}
at level $k$. 
Notice that $\mathrm{gr}_{k}\mathcal{M}_{a} = \{ \ell \in \mathcal{M}_{a} \,|\, N_{a}(\ell,j)=k_j, \, \forall j \in \mathcal{J}\}$.


We also define the generating set of holes of the graded pieces as follows. We fix $k$ and we choose vectors $\mathfrak{v}_{n}\in \mathbb{Z}\langle E_{n'}\rangle_{n'\in \mathcal{N}}$ for any $n\in \mathcal{N}$ satisfying the properties of Lemma \ref{Lm2} for the lift $a-k$ and parameter $\kappa=1$, in particular $N_{a}(\ell+ \mathfrak{v}_{n},n)\geq k_{n}+1$. Then, for every subset $\mathcal{I} \supseteq \mathcal{J}$ we set
$$
\mathrm{gr}_{k}\mathcal{M}_{a,\mathcal{I}}^{-} := \{ \ell\in (\Box - \mathbf{c}_{a-k}) \cap \mathbb{Z}^{\mathcal{N}}(a)  \,|\,  N_{a}(\ell,n) < 0 \, \, \forall n\in \mathcal{I}\setminus \mathcal{J} \, \, \, \mathrm{and} \, \, \, N_{a}(\ell,n) =k_n \, \, \forall n\in \mathcal{J}\}.
$$

The next lemma gives a rational form of the series $\mathcal{H}_{\mathcal{M}_{a}(k)}$ and $\mathcal{H}_{\mathrm{gr}_{k} \mathcal{M}_{a}}$ in terms of holes.  

\begin{lemma}\labelpar{Thm2}
\hskip1mm
\begin{enumerate}[(a)]
 \item\label{Thm2-a} 
 $\displaystyle 
 \mathcal{H}_{\mathcal{M}_{a}(k)}(\mathbf{t}_{\mathcal{N}})
 =
\sum_{\emptyset \subseteq\mathcal{I}\subseteq\mathcal{N}}(-1)^{|\mathcal{I}|}
\frac{\sum_{\ell \in 
\mathcal{M}_{a-k,\mathcal{I}}^{-}} 
\mathbf{t}^{\ell}}{ \prod_{n\notin \mathcal{I}} (1-\mathbf{t}^{\mathfrak{v}_{n}}) }$,

\item\label{Thm2-b}
$ \displaystyle
\mathcal{H}_{\mathrm{gr}_{k} \mathcal{M}_{a}}(\mathbf{t}_{\mathcal{N}})
=
\sum_{\mathcal{J} \subseteq \mathcal{I}\subseteq\mathcal{N}}(-1)^{|\mathcal{I}\setminus\mathcal{J}|}
\frac{\sum_{\ell \in \mathrm{gr}_{k}\mathcal{M}_{a,\mathcal{I}}^{-}} \mathbf{t}^{\ell}}{ \prod_{n\notin \mathcal{I}} (1-\mathbf{t}^{\mathfrak{v}_{n}}) }.$

\end{enumerate}
\end{lemma}

\begin{proof}
Remark \ref{rk:Ma(k)} and Corollary \ref{genser} applied for $\mathcal{M}_{a-k}$ deduce the formula of (\ref{Thm2-a}).
For part (\ref{Thm2-b}) we give an inclusion-exclusion formula for $\mathrm{gr}_{k}\mathcal{M}_{a}$, which will imply the formula for $\mathcal{H}_{\mathrm{gr}_{k}\mathcal{M}_{a}}(\mathbf{t}_{\mathcal{N}})$. Thus, denote $\mathcal{L}_{a-k}:=\{\ell \in \mathbb{Z}^{\mathcal{N}}(a)  \,|\,  N_{a-k}(\ell, j)\leq 0\ \forall j\in \mathcal{J}\}$ and note that $\mathrm{gr}_{k}\mathcal{M}_{a} = \mathcal{M}_{a-k} \cap \mathcal{L}_{a-k}$. The inclusion-exclusion formula for $\mathrm{gr}_{k}\mathcal{M}_{a}$ will be deduced from the following, given by the Proposition \ref{Thm1}(\ref{Thm1-b}),
$$
\mathcal{M}_{a-k} = 
\sum_{\emptyset \subseteq \mathcal{I}' \subseteq \mathcal{N}} (-1)^{|\mathcal{I}'|} \left( \mathcal{M}_{a-k, \mathcal{I}'}^{-} + \mathbb{Z}_{\geq0} \langle \mathfrak{v}_{i} \rangle_{i\notin \mathcal{I}'}\right)
$$
by intersecting it with $\mathcal{L}_{a-k}$. Therefore, we deduce the identity
\begin{equation}\label{eq:intersection}
\left( 
\mathcal{M}_{a-k, \mathcal{I}'}^{-} + \mathbb{Z}_{\geq0} \langle \mathfrak{v}_{i} \rangle_{i\notin \mathcal{I}'}
\right) \cap \mathcal{L}_{a-k} 
= 
\left( 
\mathcal{M}_{a-k, \mathcal{I}'}^{-} \cap \mathcal{L}_{a-k} \right) + \mathbb{Z}_{\geq0} \langle \mathfrak{v}_{i} \rangle_{i\notin \mathcal{I}' \cup \mathcal{J}}.
\end{equation}
Indeed, for $\ell = \ell_{0} + \sum_{i\notin \mathcal{I}'} \lambda_{i} \mathfrak{v}_{i}$ in the left hand side of (\ref{eq:intersection}) we have $0\geq N_{a-k}(\ell, j) = N_{a-k}(\ell_{0} + \lambda_{j}\mathfrak{v}_{j},j)$ for any $j\in \mathcal{J}\setminus \mathcal{I}'$, which implies $\lambda_{j}=0$ by the choice of vectors $\mathfrak{v}_{n}$, $n\in \mathcal{N}$ (cf. Lemma \ref{Lm2}(\ref{Lm2-b})). Furthermore, we have inclusion-exclusion formula for the graded holes
$$
\mathrm{gr}_{k}\mathcal{M}^{-}_{a,\mathcal{I}} 
=
\sum_{\substack{\mathcal{I}'\subseteq \mathcal{I}\\  \mathcal{I}' \cup \mathcal{J}=\mathcal{I} }} (-1)^{|\mathcal{I}'\cap \mathcal{J}|} \left( \mathcal{M}^{-}_{a-k,\mathcal{I}'} \cap \mathcal{L}_{a-k} \right). 
$$
Finally, combining the above results we get 
$$
\mathrm{gr}_{k}\mathcal{M}_{a} = \mathcal{M}_{a-k} \cap \mathcal{L}_{a-k} = \sum_{\mathcal{J} \subseteq \mathcal{I} \subseteq \mathcal{N}} (-1)^{|\mathcal{I}\setminus \mathcal{J}|} \left( \mathrm{gr}_{k}\mathcal{M}_{a,\mathcal{I}}^{-} + \mathbb{Z}_{\geq0}\langle \mathfrak{v}_{i} \rangle_{i\notin \mathcal{I}}
\right),
$$
which implies the formula for $\mathcal{H}_{\mathrm{gr}_{k}\mathcal{M}_{a}}(\mathbf{t}_{\mathcal{N}})$.
\end{proof}

\subsection[Example]{Example (part I)}\labelpar{example}
Consider the following plumbing graph $\Gamma$:
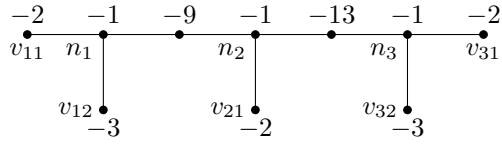
\vspace{0.2cm}
\begin{center}
\begin{tikzpicture}[scale=.5]
\coordinate (v11) at (0,0);
\draw  node[above] at (v11) {$-2$};
\draw[fill] (v11) circle (0.1);

\coordinate (v1) at (2,0);
\draw node[above] at (v1) {$-1$};
\draw[fill] (v1) circle (0.1);

\coordinate (v12) at (2,-2);
\draw   node[below] at (v12) {$-3$};
\draw[fill] (v12) circle (0.1);

\coordinate (u1) at (4,0);
\draw   node[above] at (u1) {$-9$};
\draw[fill] (u1) circle (0.1);

\coordinate (v0) at (6,0);
\draw   node[above] at (v0) {$-1$};
\draw[fill] (v0) circle (0.1);

\coordinate (v01) at (6,-2);
\draw   node[below] at (v01) {$-2$};
\draw[fill] (v01) circle (0.1);

\coordinate (u2) at (8,0);
\draw   node[above] at (u2) {$-13$};
\draw[fill] (u2) circle (0.1);

\coordinate (v2) at (10,0);
\draw   node[above] at (v2) {$-1$};
\draw[fill] (v2) circle (0.1);

\coordinate (v22) at (10,-2);
\draw   node[below] at (v22) {$-3$};
\draw[fill] (v22) circle (0.1);

\coordinate (v21) at (12,0);
\draw   node[above] at (v21) {$-2$};
\draw[fill] (v21) circle (0.1);

\draw[-] (v11) -- (v1);
\draw[-] (v12) -- (v1);
\draw[-] (u1) -- (v1);
\draw[-] (u1) -- (v0);
\draw[-] (v01) -- (v0);
\draw[-] (u2) -- (v0);
\draw[-] (u2) -- (v2);
\draw[-] (v21) -- (v2);
\draw[-] (v22) -- (v2);

\draw node[below left] at (v1) {$n_1$};
\draw node[below left] at (v0) {$n_2$};
\draw node[below left] at (v2) {$n_3$};

\draw node[below] at (v11) {$v_{11}$};
\draw node[left] at (v12) {$v_{12}$};
\draw node[left] at (v01) {$v_{21}$};
\draw node[left] at (v22) {$v_{32}$};
\draw node[below] at (v21) {$v_{31}$};
\end{tikzpicture}
\end{center}
Notice that $H$ is trivial, hence we describe the non-normal affine monoid $\mathcal{M}_{0}$ associated with $a=0$. Let $n_1,n_2,n_3$ be the nodes of $\Gamma$ and denote by $v_{11},v_{12},v_{21},v_{31},v_{32}$ the corresponding ends as shown  on the above figure. The generalized Seifert invariants are the followings: $(\alpha_{v_{11}},\omega_{v_{11}})=(2,1)$, $(\alpha_{v_{12}},\omega_{v_{12}})=(3,1)$, $(\alpha_{n_1,n_2},\omega_{n_1,n_2})=(9,1)$, $(\alpha_{v_{21}},\omega_{v_{21}})=(2,1)$, $(\alpha_{n_2,n_3},\omega_{n_2,n_3})=(13,1)$, $(\alpha_{v_{31}},\omega_{v_{31}})=(2,1)$ and $(\alpha_{v_{32}},\omega_{v_{32}})=(3,1)$. Then,   
\begin{equation*} 
\mathcal{M}_{0}
=
\left\{
\ell= \sum_{i=1}^{3}\ell_i E_{n_{i}}\in \mathbb{Z}\langle E_{i} \rangle_{i=1}^{3} \ : \ 
\begin{array}{l} 
 N_0(\ell,n_1)=\frac{8}{9}\ell_1 - \frac{1}{9}\ell_2 + \big\lfloor\frac{-\ell_1}{2} \big\rfloor + \big\lfloor \frac{-\ell_1}{3} \big\rfloor \geq 0  \\
 N_0(\ell,n_2)=\frac{95}{117}\ell_2 - \frac{1}{9}\ell_1 -\frac{1}{13}\ell_3 +\big\lfloor \frac{-\ell_2}{2} \big\rfloor \geq 0\\
 N_0(\ell,n_3)=\frac{12}{13}\ell_3 - \frac{1}{13}\ell_2 + \big\lfloor\frac{-\ell_3}{2} \big\rfloor + \big\lfloor\frac{-\ell_3}{3} \big\rfloor\geq 0\\    
 \ell_1+\ell_2 \equiv 0 \ (\mbox{mod} \ 9 )\\
 \ell_2+\ell_3 \equiv 0 \ (\mbox{mod} \ 13 )
\end{array}
\right\}. 
\end{equation*}  
Through the example we will use short notation $(\ell_{1}, \ell_{2}, \ell_{3})$ for $\ell = \ell_{1}E_{n_{1}}+\ell_{2}E_{n_{2}}+\ell_{3}E_{n_{3}}$. 
 One can take the generators $\mathfrak{v}_1=1/3\cdot\pi_{\mathcal{N}}(E_1^*)=(62,28,24)$, $\mathfrak{v}_2=\pi_{\mathcal{N}}(E_2^*)=(84,42,36)$ and $\mathfrak{v}_3=1/3\cdot\pi_{\mathcal{N}}(E_3^*)=(24,12,14)$ satisfying properties of Lemma \ref{Lm2}. They provide the following sets 
\begin{equation*}
\begin{array}{c}
\Box=\{(0,0,0),(12,6,7),(31,14,12),(42,21,18),(43,20,19),(54,27,25),(73,35,30),(85,41,37)\},\\
\mathcal{M}^-_{0,n_1}=\left\{(31,14,12), (43,20,19),(73,35,30),(85,41,37)\right\},\\
\mathcal{M}^-_{0,n_2}=\emptyset \ \ \mbox{and} \ \ 
\mathcal{M}^-_{0,n_3}=\left\{(12,6,7), (43,20,19),(54,27,25),(85,41,37)\right\}.
\end{array}
\end{equation*}
Hence, the holes of $\mathcal{M}_{0}$ are elements in the following forms:
\begin{equation*}
\begin{array}{ll}
(31,14,12)+\lambda_1\cdot(84,42,36)+\lambda_2\cdot(24,12,14), & (43,20,19)+\lambda_3\cdot(84,42,36)+\lambda_4\cdot(24,12,14),\\
(73,35,30)+\lambda_5\cdot(84,42,36)+\lambda_6\cdot(24,12,14), & (85,41,37)+\lambda_7\cdot(84,42,36)+\lambda_8\cdot(24,12,14),\\ 
(12,6,7)+\lambda_9\cdot(62,28,24)+\lambda_{10}\cdot(84,42,36), & (43,20,19)+\lambda_{11}\cdot(62,28,24)+\lambda_{12}\cdot(84,42,36),\\
(54,27,25)+\lambda_{13}\cdot(62,28,24)+\lambda_{14}\cdot(84,42,36), & (85,41,37)+\lambda_{15}\cdot(62,28,24)+\lambda_{16}\cdot(84,42,36).
\end{array}
\end{equation*}
for any $\lambda_i\in \mathbb{Z}_{\geq 0}$. The computations were performed  using Maple.

\begin{remark}
By Remark \ref{rem:basis}(\ref{rem:basis-c}) we could choose `smaller' generators, eg. one can take $\mathfrak{v}'_2:= 1/2\cdot \pi_{\mathcal{N}}(E^{*}_{2}) =  (42,21,18)$ instead of $\mathfrak{v}_{2}$. However, the generators $\mathfrak{v}_{1}, \mathfrak{v}'_{2}, \mathfrak{v}_{3}$ cannot be used for  
 Lemma \ref{Thm2}, since $N_{0}(\mathfrak{v}'_2,n_2)=0$. Notice that for our choice we have $N_{0}(\mathfrak{v}_2,n_2)=1$.
\end{remark}

\section{Rational representation of $Z_{h}(\mathbf{t}_{\mathcal{N}})$} \labelpar{ss:Zrational}

In this section we express $Z_h(\mathbf{t}_{\mathcal{N}})$ with rational functions given by the structure of the filtered $\mathcal{M}_{0}$-module $\mathcal{M}_{a}$. We assume that $\delta_{n,\mathcal{N}}\geq 1$ for every $n\in\mathcal{N}$ and consider the filtration in Section \ref{sss:Hilb} associated with the subset $\widehat{\mathcal{N}}\subset \mathcal{N}$. 

In the next theorem we give two formulas for $Z_h(\mathbf{t}_{\mathcal{N}})$. The first uses generating series of the submodules $\mathcal{M}_{a}(k)$, while the second formula is a rational representation in terms of  generating functions associated with the graded pieces $\mathrm{gr}_{k}\mathcal{M}^{-}_{a, \, \mathcal{I}}$ of the holes. This will be the core of the combinatorial formulas for the polynomial parts and Seiberg--Witten invariants given in Section \ref{ss:formulas}. 
Recall the notations $(-1)^{k}\binom{\delta-2}{k} := \prod_{n\in \widehat{\mathcal{N}}} (-1)^{k_{n}} \binom{\delta_{n,\mathcal{N}}-2}{k_{n}}$ and $0\leq k \leq \delta-2$ for $0\leq k_{n} \leq \delta_{n,\mathcal{N}}-2$ for all $n\in \widehat{\mathcal{N}}$.

\begin{thm}\label{Thm3}
Let $a$ be a reduced lift of $h\in H$. 
\begin{enumerate}[(a)]
\item\label{Thm3-a}
Let $\{\mathfrak{v}_{n}\}_{n\in \mathcal{N}}$ be a set of vectors satisfying properties of Lemma \ref{Lm2} for  every lift $a-k$ with $0\leq k \leq \delta-1$ and parameter $\kappa=0$, ie. $N_{a}(\ell+\mathfrak{v}_{n},n)\geq k_{n}$ for all $n\in \mathcal{N}$ and $\ell \in \overline{\mathcal{M}}_{a-k}$. Then
$$ Z_{h}(\mathbf{t}_{\mathcal{N}})
= 
 \sum_{0\leq k \leq \delta-1}
(-1)^{k} \binom{\delta-1}{k}
\sum_{\emptyset\subseteq\mathcal{I}\subseteq \mathcal{N}} (-1)^{|\mathcal{I}|}
\frac{\sum_{\ell\in \mathcal{M}^{-}_{a-k, \, \mathcal{I}}} 
{\mathbf{t}}^{\mathbf{c}_{a} + \ell} }{ \prod_{n\notin \mathcal{I}} (1- {\mathbf{t}}^{\mathfrak{v}_{n}}) }.
$$ 

\item\label{Thm3-b}
Let $\{\mathfrak{v}_{n}\}_{n\in \mathcal{N}}$ be a set of vectors satisfying properties of Lemma \ref{Lm2} for  every lift $a-k$ with $0\leq k \leq \delta-2$ and parameter $\kappa=1$, ie. $N_{a}(\ell+\mathfrak{v}_{n},n)\geq k_{n}+1$ for all $n\in \mathcal{N}$ and $\ell \in \overline{\mathcal{M}}_{a-k}$. Then
$$ Z_{h}(\mathbf{t}_{\mathcal{N}})= 
 \sum_{0\leq k \leq \delta-2}
(-1)^{k} \binom{\delta-2}{k}
\sum_{\widehat{\mathcal{N}} \subseteq\mathcal{I}\subseteq \mathcal{N}} (-1)^{|\mathcal{I}\setminus\widehat{\mathcal{N}}|}
\frac{\sum_{\ell\in \mathrm{gr}_{k}\mathcal{M}^{-}_{a, \, \mathcal{I}}} 
{\mathbf{t}}^{\mathbf{c}_{a} + \ell} }{ \prod_{n\notin \mathcal{I}} (1- {\mathbf{t}}^{\mathfrak{v}_{n}}) }.
$$ 
\end{enumerate}
\end{thm}
\begin{proof}
\begin{enumerate}[(a)]
\item
By the alternative decomposition (\ref{Eq:Zh-alternative}) we have 
$$
Z_{h}(\mathbf{t}_{\mathcal{N}})=\mathbf{t}^{\mathbf{c}_{a}} \cdot \sum_{0\leq k \leq \delta-1}
(-1)^{k} \binom{\delta-1}{k} \mathcal{H}_{\mathcal{M}_{a}(k)}(\mathbf{t}_{\mathcal{N}}).
$$
Applying  Lemma \ref{Thm2}(\ref{Thm2-a}),  we get the desired rational form.

\item
Note that $\displaystyle\mathcal{S}_{a} = \bigsqcup_{0\leq k \leq \delta-2}\mathrm{gr}_{k}\mathcal{M}_{a}$, thus $Z_{h}(\mathbf{t}_{\mathcal{N}}) = \sum_{0\leq k \leq \delta-2} (-1)^{k} \binom{\delta-2}{k} \mathcal{H}_{\mathrm{gr}_{k}\mathcal{M}_{a}}(\mathbf{t}_{\mathcal{N}})$ by (\ref{EqZ}). Finally, Lemma \ref{Thm2}(\ref{Thm2-b}) implies the formula.
\end{enumerate}
%
\end{proof}%

\begin{remark}
Sometimes it is simpler to choose `universal' generators $\mathfrak{v}_{n}$ of $\mathcal{M}_{0}$ satisfying the properties of Lemma \ref{Lm2}(\ref{Lm2-a}) and $\overline{N}_{0}(\mathfrak{v}_{n},n)\geq \delta_{n}-1$ for every $n\in \mathcal{N}$. In general, they are larger, but their properties are easier to check.
\end{remark}

\subsection{\bf Example (part II)} \labelpar{ex:p2}
We use all the notations and calculations from Section \ref{example}. Recall that the chosen generators of $\mathcal{M}_{0}$ are $\mathfrak{v}_1=(62,28,24)$, $\mathfrak{v}_2=(84,42,36)$ and $\mathfrak{v}_3=(24,12,14)$. 
Consider the filtration defined by Section \ref{sss:Hilb} associated with the subset $\widehat{\mathcal{N}}=\{n_2\}$. Then, Theorem \ref{Thm3} (b) implies the following rational form
\begin{align*}
Z(\mathbf{t}_{\mathcal{N}})=\frac{\sum_{\mathrm{gr}_{0}\mathcal{M}_{0,n_2}^{-}}\mathbf{t}^{\ell}}{(1-\mathbf{t}^{(62,28,24)})(1-\mathbf{t}^{(24,12,14)})}-\frac{\sum_{\mathrm{gr}_{0}\mathcal{M}_{0,\{n_1,n_2\}}^{-}}\mathbf{t}^{\ell}}{1-\mathbf{t}^{(24,12,14)}}-
\frac{\sum_{\mathrm{gr}_{0}\mathcal{M}_{0,\{n_2,n_3\}}^{-}}\mathbf{t}^{\ell}}{1-\mathbf{t}^{(62,28,24)}}+\sum_{\mathrm{gr}_{0}\mathcal{M}_{0,\{n_1,n_2,n_3\}}^{-}}\mathbf{t}^{\ell},
\end{align*}
where 
$\mathrm{gr}_{0}\mathcal{M}_{0,\mathcal{I}}^{-}=\{ \ell\in \Box  \cap \mathbb{Z}^{3}(0)  \,|\,  N_{0}(\ell,n_i) < 0 \, \, \, \forall n_i\in \mathcal{I}\setminus n_2 \, \, \, \mathrm{and} \, \, \, N_{0}(\ell,n_2) =0 \}$ 
for every $\{n_2\}\subseteq \mathcal{I} \subseteq \{n_1,n_2,n_3\}$. Thus, calculating the above associated graded sets we get  
\begin{equation*}
\begin{array}{c}
\mathrm{gr}_{0}\mathcal{M}_{0,n_2}^{-}=\{(0,0,0),(12,6,7),(31,14,12),(42,21,18),(43,20,19),(54,27,25),(73,35,30),(85,41,37)\},\\
\mathrm{gr}_{0}\mathcal{M}_{0,\{n_1,n_2\}}^{-}=\{(31,14,12),(43,20,19),(73,35,30),(85,41,37)\},\\
\mathrm{gr}_{0}\mathcal{M}_{0,\{n_2,n_3\}}^{-}=\{(12,6,7),(43,20,19),(54,27,25),(85,41,37)\} \, \, \mbox{and}\\ \mathrm{gr}_{0}\mathcal{M}_{0,\{n_1,n_2,n_3\}}^{-}=\{(43,20,19),(85,41,37)\}.
\end{array}
\end{equation*}

\section{Seiberg--Witten invariants and polynomial parts via the holes of $\mathcal{M}_{a}$} \labelpar{s:SW}

\subsection{Notations and preliminary results}

\subsubsection{\bf {Spin}$^{c}$--structures} 
Recall that $\widetilde{X}$ is a smooth 4-manifold with boundary $M$, and let $\widetilde{\sigma}_{can}$ be the canonical $spin^c$--structure on $\widetilde{X}$.  Its
first Chern class $c_1( \widetilde{\sigma}_{can})=-K\in L'$, where $K$ is the canonical element in $L'$ defined by the
adjunction equations $(K+E_v,E_v)+2=0$ for all $v\in\mathcal{V}$. 
The set of $spin^c$--structures $\mathrm{Spin}^c(\widetilde{X})$ of $\widetilde{X}$ is an $L'$--torsor: 
the $L'$--action is denoted by $l'*\widetilde{\sigma}$ and $c_1(l'*\widetilde{\sigma})=c_1(\widetilde{\sigma})+2l'$.
The $spin^c$--structures of $M$ are obtained by restrictions from $\widetilde{X}$. The set 
$\mathrm{Spin}^c(M)$ is an $H$--torsor, compatible with the restriction and the projection $L'\to H$. Hence, for any $\sigma\in \mathrm{Spin}^c(M)$ one has $\sigma=h*\sigma_{can}$ for some $h\in H$, where 
the canonical $spin^c$--structure $\sigma_{can}$ of $M$ is the restriction 
of the canonical $spin^c$--structure $\widetilde{\sigma}_{can}$ of $\widetilde{X}$ (see eg. \cite[p.\,415]{GS}).

We denote by $\mathfrak{sw}_{\sigma}(M)$ the Seiberg--Witten invariants of $M$ indexed by the $spin^c$-structures $\sigma\in \mathrm{Spin}^c(M)$. 

\subsubsection{}\labelpar{ss:order} Recall that
we have defined the Lipman cone $\mathcal{S}':=\{l'\in L'\,:\, (l',E_v)\leq 0 \ \mbox{for all $v$}\}$ which is generated over $\mathbb{Z}_{\geq 0}$ by the duals $E_v^*$. Moreover, for any $l_1,l_2\in L\otimes \mathbb{Q}$ one can say $l_1\geq l_2$ if $l_1-l_2=\sum \lambda_vE_v$ with all $\lambda_v\geq 0$. 
Then, for any fixed $x\in L'$ the set
\begin{equation}\label{eq:finite}
\{l'\in \mathcal{S}'\,:\, l'\ngeq x\} \ \ \mbox{is finite},
\end{equation}
since  all the entries  of $E_v^*$ are strictly positive. 

\subsubsection{\bf Counting functions and Seiberg--Witten invariants}\label{ss:sw}
For any $h\in H$ we define the {\it counting function} of the coefficients of  $Z_{h}(\mathbf{t})=\sum_{[l']=h}p_{l'}\mathbf{t}^{l'}$  by 
$$x\mapsto Q_{h}(x):=\sum_{l'\not\geq x,\, [l']=h} \, p_{l'}.$$
It is well defined, since the finiteness of above sum follows  
from (\ref{eq:finite}). 

By a powerful result of N\'emethi \cite{NJEMS} we know that if $x\in -K+ \textnormal{int}(\mathcal{S}')$, $[x]=h$ then 
\begin{equation}
Q_{h}(x)=-\frac{(K+2x)^2+|\mathcal{V}|}{8}-\mathfrak{sw}_{-h*\sigma_{can}}(M).
\end{equation}
In other words, $Q_{h}(x)$ is a multivariable quadratic polynomial on $L$ with constant term 
\begin{equation}\label{swnorm}
\mathfrak{sw}^{norm}_h(M):=-\frac{(K+2r_h)^2+|\mathcal{V}|}{8}-\mathfrak{sw}_{-h*\sigma_{can}}(M),
\end{equation}
which is called the {\em normalized Seiberg--Witten invariant} of $M$ associated with $h\in H$ (cf. \cite{Osz,NOSZ,NJEMS}). 
Furthermore, \cite{LN} gives the following interpretation: there exists a conical chamber decomposition of the real cone $\mathcal{S}'_{\mathbb{R}}=\cup_{\tau}\mathcal{C}_{\tau}$, a sublattice $\widetilde L\subset L$ and $l'_* \in \mathcal{S}'$ such that $Q_h(l')$ is a polynomial on  
$\widetilde L\cap(l'_* +\mathcal{C_{\tau}})$, say $Q^{\mathcal{C}_{\tau}}_h(l')$. This allows to define the {\em multivariable periodic constant} by $\mathrm{pc}^{\mathcal{C}_{\tau}}(Z_h):= Q^{\mathcal{C}_{\tau}}_h(0)$ 
associated with $h\in H$ and $\mathcal{C}_{\tau}$. Moreover, $Z_h(\mathbf{t})$ is rather special in the sense that all $Q^{\mathcal{C}_{\tau}}_h$ are equal for any $\mathcal{C}_{\tau}$. In particular, we say that there exists the periodic constant  
$\mathrm{pc}^{S'_{\mathbb{R}}}(Z_h):=\mathrm{pc}^{\mathcal{C}_{\tau}}(Z_h)$ associated with $S'_{\mathbb{R}}$, and in fact, it is equal with $\mathfrak{sw}^{norm}_h(M)$. 

\subsubsection{\bf Reduction}\labelpar{ss:reductionb}
As we have mentioned in Section \ref{ss:redtps}, \cite[Theorem 5.4.2]{LN} proves that, from Seiberg--Witten invariant point of view, the number of variables can be reduced to $|\mathcal{N}|$. 
Thus, there exists the periodic constant of $Z_h(\mathbf{t}_{\mathcal{N}})$ associated with the projected real Lipman cone  
$\pi_{\mathcal{N}}(S'_{\mathbb{R}})$ and 
$$\mathrm{pc}^{\pi_{\mathcal{N}}(S'_{\mathbb{R}})}(Z_h(\mathbf{t}_{\mathcal{N}}))=\mathrm{pc}^{S'_{\mathbb{R}}}(Z_h(\mathbf{t}))=
\mathfrak{sw}^{norm}_h(M).$$

\subsubsection{\bf Decomposition into `polynomial and negative degree' parts \cite{LSz}}\labelpar{ss:polSW} 
Lemma 22 and 23 of \cite{LSz} guarantee the existence of the following unique decompositions:
\begin{enumerate}[(i)]
\item 
$Z_h(\mathbf{t}_{\mathcal{N}})=P^n_h(\mathbf{t}_{\mathcal{N}})+Z^{n,-}_h(\mathbf{t}_{\mathcal{N}})$ with respect to the variable $t_n$, associated with any $n\in\mathcal{N}$,
\item 
$Z_h(\mathbf{t}_{\mathcal{N}})=P^{nn'}_h(\mathbf{t}_{\mathcal{N}})+Z^{nn',-}_h(\mathbf{t}_{\mathcal{N}})$ with respect to variables $t_n$ and $t_{n'}$  associated with any chain $(n,n')$ connecting the neighbouring nodes $n$ and $n'$.
\end{enumerate}
$P^n_h$ and $P^{n,n'}_h$ are the {\em polynomial parts}, while $Z^{n,-}_h$ and $Z^{nn',-}_h$ are the {\em `negative degree parts'} of the decomposition. The last appelation is inherited from the first case  where $Z^{n,-}_h$ has negative degree with respect to variable $t_{n}$.

These special decompositions associated with the nodes $n\in \mathcal{N}$ and the chains $(n,n')$ induce a unique decomposition 
$$Z_h(\mathbf{t}_{\mathcal{N}})=P_h(\mathbf{t}_{\mathcal{N}})+Z^{-}_h(\mathbf{t}_{\mathcal{N}})$$
with the following special properties:
\begin{enumerate}[(i)]
\item 
$P_h(\mathbf{t}_{\mathcal{N}})$ is a Laurent polynomial supported on $\pi_{\mathcal{N}}(L')\setminus \mathbb{Q}_{<0}\langle E_n\rangle_{n\in \mathcal{N}}$,
\item 
$P_h(1)=\textnormal{pc}^{\pi_{\mathcal{N}}(S_{\mathbb{R}}')}(Z_h(\mathbf{t}_{\mathcal{N}}))=\mathfrak{sw}^{norm}_h(M)$ 
and 
\begin{equation}\label{eq:polypart}
P_{h}(\mathbf{t}_{\mathcal{N}}) = \sum_{n<n'}P_{h}^{nn'}(\mathbf{t}_{\mathcal{N}}) - \sum_{n\in\mathcal{N}}(\delta_{n,\mathcal{N}}-1) P^{n}_{h}(\mathbf{t}_{\mathcal{N}}),
\end{equation}
\item 
$\textnormal{pc}^{\pi_{\mathcal{N}}(S_{\mathbb{R}}')}(Z^{-}_h(\mathbf{t}_{\mathcal{N}}))=0$ and 
$Z_{h}^{-}(\mathbf{t}_{\mathcal{N}}) = \sum_{n<n'}Z_{h}^{nn',-}(\mathbf{t}_{\mathcal{N}}) - \sum_{n\in\mathcal{N}}(\delta_{n,\mathcal{N}}-1) Z_{h}^{n,-}(\mathbf{t}_{\mathcal{N}}).$
\end{enumerate}

\subsection{\bf Polynomial and numerical datas associated with $\mathcal{M}_{a}$}\labelpar{ss:poldata}

In the sequel, for any $h\in H$ we choose a lift $a$ so that $\mathbf{c}_{a}\in \sum_{n\in \mathcal{N}}[0,1) E_{n}$, ie. it is uniquely determined by $h$. We fix the generators $\{\mathfrak{v}_{n}\}_{n\in \mathcal{N}}$ as in Theorem \ref{Thm3}. 

For any $\ell\in \overline{\mathcal{M}}_{a}\subset\mathbb{Z}^{\mathcal{N}}$ and $\mathcal{I}\subseteq\mathcal{N}$ we associate the generating function
$$\mathcal{H}_{(\ell,\mathcal{I})}(\mathbf{t}_{\mathcal{N}}):=\frac{\mathbf{t}^{\ell}}{\prod_{\widetilde{n}\in\mathcal{I}} (1-\mathbf{t}^{\mathfrak{v}_{\widetilde{n}}})}.$$
Then one can associate a polynomial $Pol_{(\ell,\mathcal{I})}^n(\mathbf{t}_{\mathcal{N}})$ with any $n\in\mathcal{N}$, which is defined as the polynomial part of the function $\mathcal{H}_{(\ell,\mathcal{I})}$ with respect to $t_{n}$, given by the unique decomposition proved in  \cite[Lemma 7.0.2]{BN} (see also \cite[Lemma 22]{LSz}). The explicit form of the polynomial can be deduced by division with remainder with respect to the variable $t_{n}$ (other variables are considered as coefficients): we write uniquely 
\begin{equation}\label{eq:onevar}
\mathbf{t}^{\ell}_{\mathcal{N}}=Pol^{n}_{(\ell,\mathcal{I})}(\mathbf{t}_{\mathcal{N}})\cdot\prod_{\widetilde{n}\in\mathcal{I}} (1-\mathbf{t}^{\mathfrak{v}_{\widetilde{n}}})+R^{n}(\mathbf{t}_{\mathcal{N}})
\end{equation}
such that $R^{n}$ is supported on $\sum_{\widetilde{n}\in\mathcal{I}}[0,1)\pi_n(\mathfrak{v}_{\widetilde{n}})$ as a one-variable polynomial in $t_n$. We define $pc^n_{(\ell,\mathcal{I})}:=Pol^n_{(\ell,\mathcal{I})}(1)$.

We also consider the polynomial part $Pol^{nn'}_{(\ell,\mathcal{I})}$ of $\mathcal{H}^{(\ell,\mathcal{I})}$ associated with a chain  $(n,n')$. Removing $(n,n')$ from $\Gamma$ we get two subgraphs $\Gamma_n$ and $\Gamma_{n'}$ containing $n$ and $n'$, respectively. Let $\mathcal{N}(\Gamma_n)$ and $\mathcal{N}(\Gamma_{n'})$ be the set of nodes of the respective subgraphs. Then, there are $\alpha,\beta\in \mathbb{R}_{>0}\langle E_{n},E_{n'}\rangle$ such that $\pi_{n,n'}(\mathfrak{v}_{\widetilde{n}})$ for $\widetilde{n}\in\mathcal{I}\cap\mathcal{N}(\Gamma_n)$ and $\pi_{n,n'}(\mathfrak{v}_{\widetilde{n}})$ for $\widetilde{n}\in\mathcal{I}\cap\mathcal{N}(\Gamma_{n'})$ are positive integral multiples of the vectors $\alpha$ and $\beta$, respectively. This condition together with \cite[Lemma 23]{LSz} guarantee the unique decomposition 
\begin{align}\label{eq:twovar}
\mathbf{t}^{\ell}_{\mathcal{N}} 
=& 
Pol^{nn'}_{(\ell,\mathcal{I})}(\mathbf{t}_{\mathcal{N}})\cdot\prod_{\widetilde{n}\in\mathcal{I}} (1-\mathbf{t}^{\mathfrak{v}_{\widetilde{n}}})
+
R^{nn'}_{1}(\mathbf{t}_{\mathcal{N}})\cdot\prod_{\widetilde{n}\in\mathcal{I}\cap\mathcal{N}({\Gamma_{n'}})} (1-\mathbf{t}^{\mathfrak{v}_{\widetilde{n}}})
\\
&+
R^{nn'}_{2}(\mathbf{t}_{\mathcal{N}})\cdot\prod_{\widetilde{n}\in\mathcal{I}\cap\mathcal{N}({\Gamma_{n}})} (1-\mathbf{t}^{\mathfrak{v}_{\widetilde{n}}})
+
R^{nn'}(\mathbf{t}_{\mathcal{N}})\nonumber
\end{align}
with the following conditions:
\begin{enumerate}[(i)]
\item 
$Pol^{nn'}_{(\ell,\mathcal{I})}(\mathbf{t}_{\mathcal{N}})$ is supported on $\mathbb{Z}^2\setminus\mathbb{Z}_{<0}\langle E_{n},E_{n'}\rangle$,

\item
$R^{nn'}_{1}(\mathbf{t}_{\mathcal{N}})$ is supported on $\mathbb{Z}^2 \cap \big( \mathbb{Z}_{\leq 0 } \langle E_{n} \rangle + \sum_{\widetilde{n}\in\mathcal{I}\cap\mathcal{N}({\Gamma_{n}})} [0,1)\pi_{n,n'}(\mathfrak{v}_{\widetilde{n}})\big)$,

\item
$R^{nn'}_{2}(\mathbf{t}_{\mathcal{N}})$ is supported on $\mathbb{Z}^2 \cap \big( \mathbb{Z}_{\leq 0 } \langle E_{n'} \rangle + \sum_{\widetilde{n}\in\mathcal{I}\cap\mathcal{N}({\Gamma_{n'}})} [0,1)\pi_{n,n'}(\mathfrak{v}_{\widetilde{n}})\big)$ and

\item
$R^{nn'}(\mathbf{t}_{\mathcal{N}})$ is supported on $\mathbb{Z}^2 \cap \big(  \sum_{\widetilde{n}\in\mathcal{I}} [0,1)\pi_{n,n'}(\mathfrak{v}_{\widetilde{n}}) \big)$
\end{enumerate}
as two-variable polynomial in $t_n$ and $t_{n'}$. We define the numerical data $pc^{nn'}_{(\ell,\mathcal{I})}:=Pol^{nn'}_{(\ell,\mathcal{I})}(1)$ as well. 
\begin{remark}
Notice that $pc^n_{(\ell,\mathcal{I})}$ is the periodic constant (of the Taylor expansion) of the one-variable function $\mathcal{H}^{(\ell,\mathcal{I})}\mid_{t_{n'}=1,n'\neq n}$,
while $pc^{nn'}_{(\ell,\mathcal{I})}$ 
is the periodic constant of (the Taylor expansion of) the two-variable function $\mathcal{H}^{(\ell,\mathcal{I})}\mid_{t_{n''}=1,n''\neq n,n'}$ associated with the chamber $\mathbb{R}_{>0}\langle \alpha,\beta\rangle$ (cf. \cite{LN}).
\end{remark}

Finally, we consider the polynomial
\begin{equation}\label{eq:Pol(l,I)}
 Pol_{(\ell,\mathcal{I})}(\mathbf{t}_{\mathcal{N}}):=\sum_{n<n'} Pol^{nn'}_{(\ell,\mathcal{I})}(\mathbf{t}_{\mathcal{N}})-\sum_{n\in\mathcal{N}}(\delta_{n,\mathcal{N}}-1)Pol^{n}_{(\ell,\mathcal{I})}(\mathbf{t}_{\mathcal{N}})
\end{equation}
and the constant $pc_{(\ell,\mathcal{I})}:=Pol_{(\ell,\mathcal{I})}(1)$ associated with the pair $(\ell,\mathcal{I})$.

\subsection{\bf Formulas for polynomial parts and Seiberg--Witten invariants}
\labelpar{ss:formulas} 
The decomposition of $Z_h$ presented in Section \ref{ss:polSW} together with Section \ref{ss:poldata} and Theorem \ref{Thm3}(b) imply the following formulas

\begin{cor}\labelpar{cor:pol}
$$P_h(\mathbf{t}_{\mathcal{N}})= 
\sum_{\widehat{\mathcal{N}}\subseteq\mathcal{I}\subseteq \mathcal{N}} (-1)^{|\mathcal{I}\setminus\widehat{\mathcal{N}}|} \sum_{0\leq k \leq \delta-2}
(-1)^{k} \binom{\delta-2}{k}
\sum_{\ell\in \mathrm{gr}_{k}\mathcal{M}^{-}_{a, \, \mathcal{I}}} 
Pol_{(\ell,\mathcal{N}\setminus \mathcal{I})}(\mathbf{t}_{\mathcal{N}}).$$ 
\end{cor}
In particular, one gets 
\begin{cor} \labelpar{cor:swnorm}
$$\mathfrak{sw}^{norm}_h(M)= 
\sum_{\widehat{\mathcal{N}}\subseteq\mathcal{I}\subseteq \mathcal{N}} (-1)^{|\mathcal{I}\setminus\widehat{\mathcal{N}}|} \sum_{0\leq k \leq \delta-2}
(-1)^{k} \binom{\delta-2}{k}
\sum_{\ell\in \mathrm{gr}_{k}\mathcal{M}^{-}_{a, \, \mathcal{I}}} 
pc_{(\ell,\mathcal{N}\setminus \mathcal{I})}.$$ 
\end{cor}
\begin{remark}
In fact, when $\Gamma$ has only two nodes, the previous formula agrees with the combinatorial formula from \cite[Corollary 7.4.2]{LN}.
\end{remark}

\subsection{\bf Example (part III)}\labelpar{ex:p3} 
According to the decompositions (\ref{eq:onevar}) and (\ref{eq:twovar}) and definition (\ref{eq:Pol(l,I)}), the polynomial datas which contribute to the polynomial part of $Z(\mathbf{t}_{\mathcal{N}})$ are the followings:
\begin{equation*}
\begin{array}{c}
Pol_{((85,41,37),\{n_1,n_3\})}(\mathbf{t}_{\mathcal{N}})=\mathbf{t}^{(-1,1,-1)}; \ \ \ Pol_{((85,41,37),\{n_3\})}(\mathbf{t}_{\mathcal{N}})=-\mathbf{t}^{(61,29,23)}-\mathbf{t}^{(37,17,9)}-\mathbf{t}^{(13,5,-5)};\\
Pol_{((31,14,12),\{n_3\})}(\mathbf{t}_{\mathcal{N}})=-\mathbf{t}^{(7,2,-2)}; \ \ \ Pol_{((73,35,30),\{n_3\})}(\mathbf{t}_{\mathcal{N}})=-\mathbf{t}^{(49,23,16)}-\mathbf{t}^{(25,11,2)}-\mathbf{t}^{(1,-1,-12)};\\ 
Pol_{((43,20,19),\{n_3\})}(\mathbf{t}_{\mathcal{N}})=-\mathbf{t}^{(19,8,5)}; \ \ \ Pol_{((54,27,25),\{n_1\})}(\mathbf{t}_{\mathcal{N}})=-\mathbf{t}^{(-8,-1,1)}; \\  
Pol_{((85,41,37),\{n_1\})}(\mathbf{t}_{\mathcal{N}})=-\mathbf{t}^{(23,13,13)};\  
Pol_{((43,20,19),\emptyset)}(\mathbf{t}_{\mathcal{N}})=-\mathbf{t}^{(43,20,19)};\\ 
Pol_{((85,41,37),\emptyset)}(\mathbf{t}_{\mathcal{N}})=\mathbf{t}^{(85,41,37)}.
\end{array}
\end{equation*}
Therefore, Corollary \ref{cor:pol} deduces 
$P(\mathbf{t}_{\mathcal{N}})=\mathbf{t}^{(-1,1,-1)}+\mathbf{t}^{(61,29,23)}+\mathbf{t}^{(37,17,9)}+\mathbf{t}^{(13,5,-5)}+\mathbf{t}^{(7,2,-2)}+\mathbf{t}^{(49,23,16)}+\mathbf{t}^{(25,11,2)}+\mathbf{t}^{(1,-1,-12)}+\mathbf{t}^{(19,8,5)}+\mathbf{t}^{(-8,-1,1)}+\mathbf{t}^{(23,13,13)}+\mathbf{t}^{(43,20,19)}+\mathbf{t}^{(85,41,37)}$. In particular, one gets $\mathfrak{sw}^{norm}(M)=13$.

One can compare the above calculations with \cite[Example 6.5]{LSz}.

\subsection{\bf Example with nontrivial $H$}\labelpar{ex:nontriv}
Let us modify the decorations of the graph from the above example and consider the following 
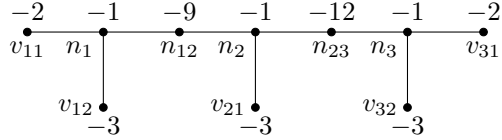
\vspace{0.2cm}
\begin{center}
\begin{tikzpicture}[scale=.5]
\coordinate (v11) at (0,0);
\draw  node[above] at (v11) {$-2$};
\draw[fill] (v11) circle (0.1);

\coordinate (v1) at (2,0);
\draw node[above] at (v1) {$-1$};
\draw[fill] (v1) circle (0.1);

\coordinate (v12) at (2,-2);
\draw   node[below] at (v12) {$-3$};
\draw[fill] (v12) circle (0.1);

\coordinate (u1) at (4,0);
\draw   node[above] at (u1) {$-9$};
\draw[fill] (u1) circle (0.1);

\coordinate (v0) at (6,0);
\draw   node[above] at (v0) {$-1$};
\draw[fill] (v0) circle (0.1);

\coordinate (v01) at (6,-2);
\draw   node[below] at (v01) {$-3$};
\draw[fill] (v01) circle (0.1);

\coordinate (u2) at (8,0);
\draw   node[above] at (u2) {$-12$};
\draw[fill] (u2) circle (0.1);

\coordinate (v2) at (10,0);
\draw   node[above] at (v2) {$-1$};
\draw[fill] (v2) circle (0.1);

\coordinate (v22) at (10,-2);
\draw   node[below] at (v22) {$-3$};
\draw[fill] (v22) circle (0.1);

\coordinate (v21) at (12,0);
\draw   node[above] at (v21) {$-2$};
\draw[fill] (v21) circle (0.1);

\draw[-] (v11) -- (v1);
\draw[-] (v12) -- (v1);
\draw[-] (u1) -- (v1);
\draw[-] (u1) -- (v0);
\draw[-] (v01) -- (v0);
\draw[-] (u2) -- (v0);
\draw[-] (u2) -- (v2);
\draw[-] (v21) -- (v2);
\draw[-] (v22) -- (v2);

\draw node[below left] at (v1) {$n_1$};
\draw node[below left] at (v0) {$n_2$};
\draw node[below left] at (v2) {$n_3$};
\draw node[below] at (u1) {$n_{12}$};
\draw node[below] at (u2) {$n_{23}$};

\draw node[below] at (v11) {$v_{11}$};
\draw node[left] at (v12) {$v_{12}$};
\draw node[left] at (v01) {$v_{21}$};
\draw node[left] at (v22) {$v_{32}$};
\draw node[below] at (v21) {$v_{31}$};
\end{tikzpicture}
\end{center}
Computing the determinant of the graph we find that $H$ has order $9$. Moreover, $g_{v}=0$ for $v\in \{v_{11}, n_{1},n_{2},n_{3},v_{31}\}$ since corresponding rows of $I^{-1}$ has integer entries, ie. $E^{*}_{v}\in L$. Furthermore, we can compute the following relations for the remaining generators $g_{n_{12}}=-g_{v_{12}}$, $g_{n_{23}}=-g_{v_{32}}$ and $3g_{v_{12}}=3g_{v_{32}}=0$, whence $H\simeq \mathbb{Z}_{3}\times \mathbb{Z}_{3}$. Let us choose the group element $h=g_{v_{12}}+g_{v_{32}}$ and $a=E^{*}_{v_{12}}+E^{*}_{v_{32}}$ as its reduced lift to demonstrate the computation of $P_{h}(\mathbf{t}_{\mathcal{N}})$ and $\mathfrak{sw}^{norm}_{h}(M)$.

We will use through this example the short notation $(\ell_{1},\ell_{2},\ell_{3})$ for $\ell = \ell_{1}E_{n_{1}} + \ell_{2} E_{n_{2}}+ \ell_{3} E_{n_{3}}$. 
The generalized Seifert invariants are $(\alpha_{v_{11}},\omega_{v_{11}})=(2,1)$, $(\alpha_{v_{12}},\omega_{v_{12}})=(3,1)$, $(\alpha_{n_1,n_2},\omega_{n_1,n_2})=(9,1)$, $(\alpha_{v_{21}},\omega_{v_{21}})=(3,1)$, $(\alpha_{n_2,n_3},\omega_{n_2,n_3})=(12,1)$, $(\alpha_{v_{31}},\omega_{v_{31}})=(2,1)$ and $(\alpha_{v_{32}},\omega_{v_{32}})=(3,1)$.
By Remark \ref{rem:basis}(\ref{rem:basis-c}) and Theorem \ref{Thm3}  we can consider the basis $\mathfrak{v}_{1}:=\pi_{\mathcal{N}}(E^{*}_{n_{1}})/2 =(21,6,6) $, $\mathfrak{v}_{2}:=\pi_{\mathcal{N}}(E^{*}_{n_{2}}) =(12,6,6)$ and $\mathfrak{v}_{3}:=\pi_{\mathcal{N}}(E^{*}_{n_{3}})/2 =(6,3,9)$. 

First, we compute $\mathcal{M}^{-}_{a, \emptyset} = (\Box -\mathbf{c}_{a})\cap \mathbb{Z}^{\mathcal{N}}(a)$, where $\mathbf{c}_{a} = (18,6,10)$ and $\mathbb{Z}^{\mathcal{N}}(a) = \{\ell =\sum_{i=1}^{3} \ell_{i}E_{n_{i}} \in \mathbb{Z}\langle E_{n_{i}}\rangle_{i=1}^{3}  \,|\, \ell_{1}+\ell_{2} \equiv 0\ (\textnormal{mod }9),\ \ell_{2}+\ell_{3} \equiv 0\ (\textnormal{mod }12)\}$. Since $\mathbf{c}_{a} = \pi_{\mathcal{N}}(E^{*}_{n_{1}} )/3+\pi_{\mathcal{N}}(E^{*}_{n_{3}} )/3$ and by the choice of the basis $\{\mathfrak{v}_{i} \}_{i}$ we have
$$
\mathcal{M}^{-}_{a, \emptyset} 
= 
\left\{ \ell \in \mathbb{Z}^{\mathcal{N}}(a)  \,:\!  
\begin{array}{l} 
0\leq \frac{1}{3} + \frac{\ell_{1}}{18} - \frac{\ell_{2}}{9} < \frac{1}{2} 
\\
0\leq -\frac{\ell_{1}}{9} + \frac{17\ell_{2}}{36} - \frac{\ell_{3}}{12} < 1 
\\
0\leq \frac{1}{3} - \frac{\ell_{2}}{12} + \frac{\ell_{3}}{12} < \frac{1}{2} 
 \end{array} \!\!
 \right\}
 =
\left\{ \ell \in \mathbb{Z}^{\mathcal{N}}(a)  \,:\!  
\begin{array}{l} 
-6 \leq \ell_{1} - 2\ell_{2}< 3 
\\
0\leq -4\ell_{1} + 17\ell_{2} - 3\ell_{3} < 36 
\\
-4 \leq  - \ell_{2} + \ell_{3} < 2 
 \end{array} \!\!
 \right\}. 
$$
Combining inequalities we get $-18\leq \ell_{1} <21$, $-6\leq \ell_{2} \leq 9$, $-10\leq \ell_{3}<11$, hence $-24\leq \ell_{1}+\ell_{2} <30$ and $-16 \leq \ell_{2}+\ell_{3}<20$.  Since $\ell \in \mathbb{Z}^{\mathcal{N}}(a)$, the possible values are $\ell_{1}+\ell_{2} \in \{-18, 9, 0, 9, 18, 27\}$ and $\ell_{2}+\ell_{3} \in \{-12, 0, 12\}$, respectively. The inequalities $-6\leq \ell_{1}-2\ell_{2} < 3$ and $-4\leq \ell_{3} - \ell_{2} <2 $ restrict the possible triplets $\ell=(\ell_{1}, \ell_{2}, \ell_{3})$ to $(-12,-6,-6)$, $(-13,-5,-7)$, $(-14,-4,-8)$, $(0,0,0)$, $(-1,1,-1)$, $(-2,2,-2)$, $(12,6,6)$, $(11,7,5)$ and $(10,8,4)$. However, $\mathcal{M}^{-}_{a,\emptyset}=\{(-14,-4,-8),\, (0,0,0),\, (-1,1,-1)\}$, since these are the only triplets satisfying $0\leq -4\ell_{1}+17\ell_{2}-3\ell_{3}<36$. 
Denote 
$N_{a}(\ell) := (N_{a}(\ell,n_{1}), N_{a}(\ell,n_{2}), N_{a}(\ell,n_{3})) = 
( \frac{8\ell_{1}}{9} - \frac{\ell_{2}}{9}+ \lfloor \frac{-\ell_{1}}{2} \rfloor + \lfloor \frac{1-\ell_{1}}{3} \rfloor,\ 
-\frac{\ell_{1}}{9}+\frac{29\ell_{2}}{36}- \frac{\ell_{3}}{12} +\lfloor \frac{-\ell_{2}}{3} \rfloor,\ 
 -\frac{\ell_{2}}{12}+\frac{11\ell_{3}}{12} + \lfloor \frac{-\ell_{3}}{2} \rfloor + \lfloor \frac{1 - \ell_{3}}{3} \rfloor   
)$. Then we have $N_{a}(-14,-4,-8)=(0,0,0)$, $N_{a}(0,0,0)=(0,0,0)$ and $N_{a}(-1,1,-1)=(-1,0,-1)$, hence we can easily see the graded pieces 
\begin{align*}
\mathrm{gr}_{0}\mathcal{M}_{0,n_2}^{-} =\mathcal{M}^{-}_{a,\emptyset}, \, \, \ \
\mathrm{gr}_{0}\mathcal{M}_{0,\{n_1,n_2\}}^{-}=\mathrm{gr}_{0}\mathcal{M}_{0,\{n_2,n_3\}}^{-}=\mathrm{gr}_{0}\mathcal{M}_{0,\{n_1,n_2,n_3\}}^{-}=\{(-1,1,-1)\}.
\end{align*}
This gives the rational form of the topological Poincar\'e series associated with $h$
\begin{align*}
Z_{h}(\mathbf{t}_{\mathcal{N}})=\frac{\mathbf{t}^{(4,2,2)}+\mathbf{t}^{(17,7,9)}+\mathbf{t}^{(18,6,10)}}{(1-\mathbf{t}^{(21,6,6)})(1-\mathbf{t}^{(6,3,9)})}-\frac{\mathbf{t}^{(17,7,9)}}{1-\mathbf{t}^{(6,3,9)}}-
\frac{\mathbf{t}^{(17,7,9)}}{1-\mathbf{t}^{(21,6,6)}}+\mathbf{t}^{(17,7,9)}. 
\end{align*}
By Section \ref{ss:poldata} the polynomial datas appearing above are 
$Pol_{((17,7,9),\{n_3\})}(\mathbf{t}_{\mathcal{N}})=-\mathbf{t}^{(11,4,0)}-\mathbf{t}^{(5,1,-9)}$,  
$Pol_{((17,7,9),\{n_1\})}(\mathbf{t}_{\mathcal{N}})=-\mathbf{t}^{(-4,1,3)}$, 
$Pol_{((17,7,9),\emptyset)}(\mathbf{t}_{\mathcal{N}})=\mathbf{t}^{(17,7,9)}$.
Hence, Corollary \ref{cor:pol} and \ref{cor:swnorm} give that 
\begin{align*}
P_{h}(\mathbf{t}_{\mathcal{N}})=\mathbf{t}^{(17,7,9)}+\mathbf{t}^{(11,4,0)}+\mathbf{t}^{(5,1,-9)}+\mathbf{t}^{(-4,1,3)} \ \ \, \mbox{and} \ \ \,  \mathfrak{sw}^{norm}_{h}(M)=4.
\end{align*}

\section{Examples and analogies}\labelpar{s:exanal}

\subsection{Semigroup of plane curve singularities}\labelpar{sec:SpcI}

\subsubsection{} 
Let $K\subset S^3$ be an algebraic knot, ie. the link of an irreducible plane curve singularity defined by a germ of function $g:(\mathbb{C}^2,0)\longrightarrow(\mathbb{C},0)$. Then $K$ is an iterated torus knot and any of the well known invariants -- semigroup, linking pairs, Newton pairs, Alexander polynomial, embedded resolution graph, or equivalently, plumbing graph of $K$ -- characterizes completely the isotopy type of $K\subset S^3$. For more details see \cite{BKcurves} and \cite{EN} as general references for the theory.

The set of intersection multiplicities of $g$ with all possible analytic germs is a semigroup and it will be denoted by $\mathcal{M}_g$. Although the definition is analytic, one can describe $\mathcal{M}_g$ combinatorially by giving its Hilbert basis in terms of the linking pairs $(p_i,a_i)_{i=1}^r$: 
\begin{equation}\label{Hilbsem}
\mathrm{Hilb}(\mathcal{M}_g):=\{p_1 p_2\dots p_r, \ a_i p_{i+1}\dots p_r, \ a_r     \ |\  i=1,\dots,r-1\}. 
\end{equation} 
We consider the minimal embedded resolution graph of the singularity, or equivalently, the minimal negative definite plumbing graph $\Gamma_g$ of the knot $K\subset S^3$. $\Gamma_g$ is a tree with an extra arrow representing the knot, and it has the following shape
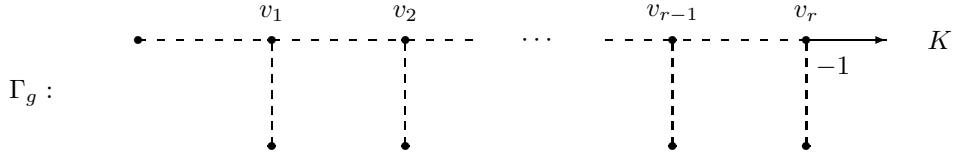
\begin{center}
\begin{picture}(400,80)(-20,0)
\put(50,60){\circle*{3}} \put(100,60){\circle*{3}}
\put(150,60){\circle*{3}} \put(250,60){\circle*{3}}
\put(300,60){\circle*{3}} \put(100,20){\circle*{3}}
\put(150,20){\circle*{3}} \put(250,20){\circle*{3}}
\put(300,20){\circle*{3}} \put(350,60){\makebox(0,0){$K$}}
\dashline{3}(50,60)(175,60) \dashline{3}(225,60)(300,60)
\put(100,20){\dashbox{3}(0,40){}}
\put(150,20){\dashbox{3}(0,40){}}
\put(250,20){\dashbox{3}(0,40){}}
\put(300,20){\dashbox{3}(0,40){}}
\put(200,60){\makebox(0,0){$\cdots$}}
\put(300,60){\vector(1,0){30}} 
\put(300,70){\makebox(0,0){$v_r$}}
\put(100,70){\makebox(0,0){$v_1$}}
\put(150,70){\makebox(0,0){$v_2$}}
\put(250,70){\makebox(0,0){$v_{r-1}$}}
\put(310,50){\makebox(0,0){$-1$}}
\put(10,40){\makebox(0,0){$\Gamma_g:$}}
\end{picture}
\end{center}
where the dash-lines represent the chains and legs. 
We regard the unique $(-1)$-vertex $v_r$ as a node of the graph $\Gamma_g$ and run the construction from Section \ref{s:monoid}. The result is an affine monoid of rank $r$ denoted by   
$\mathcal{M}_{\Gamma_g}$, which can be described with the linking, or, equivalently, the Newton pairs using eg. \cite{EN}. 

For simplicity, we describe $\mathcal{M}_{\Gamma_g}$ in the case when there is only one linking pair $(p,a)$. One can regard $\Gamma_g$ as a star-shaped graph with 
Seifert invariants $(p,\omega_p)$, $(a,\omega_a)$, and $(1,0)$ is attached with the arrow-leg representing the knot. Here, $\omega_p$ and $\omega_a$ are uniquely determined by the Diophantine equation $pa-\omega_p a -\omega_a p=1$. 
Thus, by Section \ref{s:monoid} we get  
$$\mathcal{M}_{\Gamma_g}=\Big\{\ell \in \mathbb{Z}_{\geq 0} \ | \ \ell - \Big\lceil\frac{\omega_p\ell}{p}\Big\rceil - \Big\lceil\frac{\omega_a\ell}{a}\Big\rceil\geq 0 \Big\}. $$ 
It can be checked that $\mathcal{M}_{\Gamma_g}$ is a numerical semigroup minimally generated by $p$ and $a$. Hence, it equals with the semigroup $\mathcal{M}_g$ by (\ref{Hilbsem}).

In the next section we will show how to recover $\mathcal{M}_g$ from the affine monoid $\mathcal{M}_{\Gamma_g}$ in the general case, when $r>1$. 

\subsubsection{\bf Semigroup, Poincar\'e series and analogy}\labelpar{sec:SpcII}

We denote the set of vertices of $\Gamma_g$ by $\mathcal{V}_g$ and the set of nodes by $\mathcal{N}_g:=\{v_1,\dots,v_r\}$. Following Section \ref{ss:TPs}, one can define associated with $\Gamma_g$ the reduced zeta function 
$f_{\Gamma_g}(\mathbf{t}_{\mathcal{N}_g})=\prod_{v\in \mathcal{V}_g}(1-\mathbf{t}_{\mathcal{N}_g}^{E_v^*})^{\delta_v-2}$. 
Notice that the dual elements $E_v^*$ are associated with the graph we get from $\Gamma_g$ by deleting the arrow (a plumbing graph of $S^3$). However, the vertex $v_{r}$ has $\delta_{v_r}=3$ which takes into account the information on the knot $K\subset S^3$ as well.  
Similar calculation as in Section \ref{ss:eqdec} shows that 
the reduced Poincar\'e series equals
$$
Z_{\Gamma_g}(\mathbf{t}_{\mathcal{N}_g})=\sum_{\ell \in \mathrm{gr}_{0}\mathcal{M}_{\Gamma_g}}\mathbf{t}^{\ell},
$$
 where the filtration on $\mathcal{M}_{\Gamma_g}$ is associated with $\mathcal{N}_g=\{v_2,\dots,v_{r-1}\}$.
Moreover, this induces the following correspondence between the semigroup $\mathcal{M}_g$ and the affine monoid $\mathcal{M}_{\Gamma_g}$. 

\begin{thm}
The projection $\pi_{v_r}(\mathrm{gr}_{0}\mathcal{M}_{\Gamma_g})\subset \mathbb{Z}_{\geq 0}$ is a numerical semigroup equal with $\mathcal{M}_g$. 
\end{thm}

\begin{proof}
The definition above and the A'Campo formula \cite{Acampo} deduce that the one-variable series $Z_{g}(t):=Z_{\Gamma_g}(\mathbf{t}_{\mathcal{N}_g})\mid_{t_v=1,v\neq v_r}=\sum_{\ell \in \mathrm{gr}_{0}\mathcal{M}_{\Gamma_g}}t^{\pi_{v_r}(\ell)}$ is the expansion of the monodromy zeta function of $g$ (see also \cite{EN}). Moreover, \cite{CDG99} implies the following identities
\begin{equation}\label{idMonZeta}
Z_{g}(t)=\mathcal{H}_{\mathcal{M}_g}(t)=\frac{\Delta(t)}{1-t},
\end{equation}
where $\mathcal{H}_{\mathcal{M}_g}(t)$ is the (fine) Hilbert series of the semigroup $\mathcal{M}_g$ and $\Delta(t)$ is the Alexander polynomial of the knot $K\subset S^3$. Thus, the projection $\pi_{v_r}:\mathrm{gr}_{0}\mathcal{M}_{\Gamma_g}\longrightarrow \pi_{v_r}(\mathrm{gr}_{0}\mathcal{M}_{\Gamma_g})$ is a bijection and $\pi_{v_r}(\mathrm{gr}_{0}\mathcal{M}_{\Gamma_g})$ agrees with $\mathcal{M}_g$.
\end{proof}

Let $P_{g}(t)$ be the polynomial part of $Z_{g}(t)$. One can compare with the Alexander polynomial via the following explicit expressions coming from (\ref{idMonZeta}):
$$P_g(t)=-\sum_{s\notin \mathcal{M}_g}t^s \ \ \  \ \mbox{and} \ \ \  \ \Delta(t)=1-\sum_{\substack{s\notin \mathcal{M}_g\\s-1\in \mathcal{M}_g}}t^s+\sum_{\substack{s\in \mathcal{M}_g\\s-1\notin \mathcal{M}_g}}t^s.$$ 
Therefore, $-P_{g}(1)$ equals the genus of the semigroup $\mathcal{M}_{g}$, which is the \emph{delta-invariant} $\delta$  of the plane curve singularity. 
Notice also that the degree of $P_g(t)$ is one less than of $\Delta(t)$  and equals $2\delta-1$. Nevertheless, two polynomials are `equivalent' in the sense that they determine the semigroup $\mathcal{M}_g$.

In summary, similarly to the case of plane curve singularities where we can associate with the semigroup the triple of invariants $(Z_g(t),P_g(t),\delta)$, this article is focused to the analogous picture for the link of normal surface singularities by associating with the modules $\mathcal{M}_{a}$ the triples of invariants $(Z_h(\mathbf{t}), P_h(\mathbf{t}), \mathfrak{sw}^{norm}_h)_{h\in H}$.

\subsection{\bf Semigroup of Seifert homology spheres}\labelpar{sec:Sf3}

Let $M$ be the Seifert homology sphere $\Sigma(\alpha_1,\dots,\alpha_d)$. Notice that $1<\alpha_1<\dots<\alpha_d$ are pairwise relatively prime integers and the triviality of $H$ implies that the Seifert  invariants are uniquely determined by the Diophantine equation $b_0\prod_{i=1}^d\alpha_i+\sum_{i=1}^d\omega_i\prod_{j=1\atop j\neq i}^d\alpha_j=-1$.

We consider the associated semigroup 
$$\mathcal{M}_{0}=\Big\{\ell \in \mathbb{Z}_{\geq 0} \ | \ -b_0\ell - \sum_{i=1}^d\Big\lceil\frac{\omega_i\ell}{\alpha_i}\Big\rceil\geq 0 \Big\},$$
given by Section \ref{s:monoid}. Can and Karakurt \cite{CK} proved that 
$$\mathcal{M}_{0}\cap [0,\alpha_1 \dots \alpha_d(d-2-\sum_{i=1}^d1/\alpha_i)]\subset \mathbb{Z}_{\geq 0}\langle(\alpha_1\dots\alpha_d/\alpha_i)_{i=1}^d\rangle$$
where the latter is the numerical semigroup minimally generated by $\alpha_1\dots\alpha_d/\alpha_i$ for $i=1,2,\dots,d$. The importance of this inclusion is the fact that the above piece of $\mathcal{M}_0$ already determines the Heegaard--Floer homology type of the manifold $\Sigma(\alpha_1,\dots,\alpha_d)$. However, the structure of $\mathcal{M}_{0}$ (as a semigroup) is more involved. These observations motivated the forthcoming article \cite{Semigp} which proves that, in fact, $\mathcal{M}_{0}=\mathbb{Z}_{\geq 0}\langle(\alpha_1\dots\alpha_d/\alpha_i)_{i=1}^d\rangle$.

\end{document}